\documentclass[leqno, 11pt]{amsart}
\usepackage[%
  tmargin=1in,bmargin=1in,%
  lmargin=1in,rmargin=1in,%
  marginparsep=0.2in, marginparwidth=1.0in%
]{geometry}
\usepackage{preamble}

\usepackage{graphicx}
\usepackage{datetime}

\usepackage[backend=biber,style=alphabetic,backref,maxbibnames=99]{biblatex}

\DefineBibliographyStrings{english}{
  backrefpage={},
  backrefpages={}
}

\usepackage{xpatch}
\DeclareFieldFormat{backrefparens}{\addperiod{#1}}
\xpatchbibmacro{pageref}{parens}{backrefparens}{}{}

\DeclareFieldFormat{postnote}{#1}%removes p. from citations

\bibliography{sources}

\begin{document}

\title{Affineness and reconstruction in higher Zariski geometry}
\author{Anish Chedalavada}
\date{\today}

\begin{abstract}
   We initiate a study of ``$2$-schemes'' within the geometric framework introduced in \href{https://arxiv.org/abs/2508.11621}{\UrlFont{arXiv:2508.11621 [math.AG]}}. In this setting, we show that a rigid $2$-scheme is affine if and only if it is quasicompact and quasiseparated. Using this criterion and the framework of \emph{relative spectra} developed by Lurie, we provide a universal property for the 2-rings of perfect complexes on qcqs spectral or Dirac spectral schemes. As an application, for a qcqs spectral or Dirac spectral scheme $X$, we construct a comparison morphism from $\Spec \Perf_{X}$ to $X$ itself which is moreover natural in $X$. This construction provides a new proof of the Balmer-Thomason reconstruction of of an ordinary scheme $X$ from the space of thick subcategories of $\Perf_{X}$, assuming the result for noetherian rings due to Neeman. As another application, we find spectral and Dirac spectral enhancements of support varieties arising for 2-rings in representation theory which ``geometrize'' the 2-rings that produce them. For example, given a finite group $G$ over a field $k$, this produces a ``spectral support variety'' $\mathcal{V}_{G}$ such that $\Perf_{\mathcal{V}_{G}}$ maps into the stable module category of $kG$. We also discuss examples arising from the theory of permutation modules, following work of Balmer-Gallauer. \end{abstract}

  \subjclass[2020]{14A20, 18F99, 18G80, 55P42, 55U35}
  \maketitle
  \thispagestyle{empty}

  \vspace{-3pt}
   \begin{figure}[h]
   \includegraphics[scale=1.3]{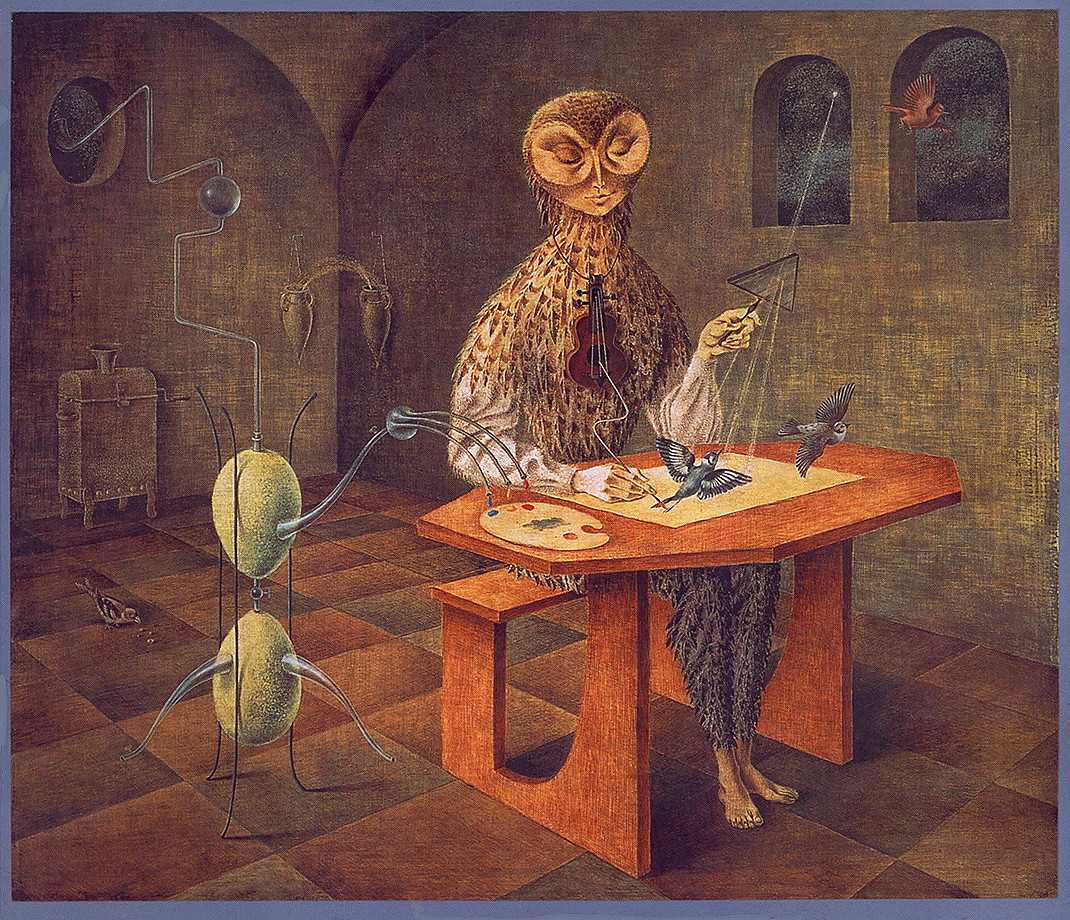}
   \caption[Cover Caption]{\emph{Creation of the Birds,} Remedios Varo, 1957\footnotemark}
 \end{figure}
\footnotetext{{\url{https://www.wikiart.org/en/remedios-varo/creation-of-the-birds}. Accessed on 10/17/2025}.}

\newpage

\setcounter{tocdepth}{1}
\tableofcontents

\section{Introduction}

\subsection{Schemes in classical algebraic geometry } Affine varieties, say over $\bb{C}$, are zero sets of collections of polynomial equations in $\bb{A}^{n}_{\bb{C}}$. Affine varieties have the decidedly nice feature that for two affine varieties $V, W$, regular maps between $V$ and $W$ may be identified with maps of commutative $\bb{C}$-algebras between $\cc{O}(W)$ and $\cc{O}(V)$, their corresponding rings of \emph{regular functions} into the base field $\bb{C}$. Classical algebraic geometry over the complex numbers deals with the study of \emph{complex varieties}, which are topological spaces admitting an open cover by affine varieties. 

Given two complex varieties $M$ and $N$, it is not necessarily the case that maps between $M$ and $N$ are determined by maps between their rings of regular functions into $\bb{C}$, although this \emph{is} the case whenever the target is affine. Since they are locally affine, it is possible to compute maps between $M$ and $N$ by ``gluing'' maps between their corresponding affine open subsets. In this way the geometry of complex varieties is rendered accessible, but being a more general class they afford the construction of interesting objects to map to, or \emph{moduli objects}, which bear import even if one is only interested in affine varieties.

\begin{example*}[Complex projective space]
  The projective space $\bb{P}^{n-1}_{\bb{C}}$ is a complex variety whose points parametrize the linear subspaces of the vector space $\bb{C}^{n}$. Maps from a variety $M$ into $\bb{P}^{n-1}_{\bb{C}}$ exactly parametrizes subspaces $E \subseteq \bb{A}_{\bb{C}}^{n} \times M$ which are fiberwise linear over $M$. 
\end{example*}

  Following Zariski-Grothendieck, let us regard the objects above not merely as topological spaces, but in their following guise.

\begin{recollection*}
  The category of \emph{locally ringed spaces} consists of pairs $(X, \cc{O}_{X})$, where $X$ is a topological space and $\cc{O}_{X}$ is a sheaf of commutative rings on $X$ satisfying a certain locality condition. Any ordinary commutative ring $R$ arises as the global sections of a structure sheaf on a particular topological space $\Spec R$, its \emph{classical Zariski spectrum}. $\Spec R$ along with its structure sheaf is the primary example of a locally ringed space, and the functor sending $R$ to the locally ringed space $\Spec R$ participates in the following adjunction: \begin{equation}\label{eq:classicalzariskispec}
  \Gamma: \{\mm{Locally~ringed~spaces}\} \rightleftarrows \{\mm{Commutative~rings}\}^{\op}: \Spec
  \end{equation}  and the essential image of the functor $\Spec$ is said to comprise exactly the \emph{affine schemes}. 
\end{recollection*}

Affine schemes are an enhancement of complex affine varieties, which can be recovered from the Zariski spectra of their rings of regular functions. In this case, the adjunction of \eqref{eq:classicalzariskispec} exactly recovers the characteristic mapping property discussed in the first paragraph. It remains to implement the study of non-affine varieties in this language, for which one makes the following definition. 

\begin{recollection*}
  Let $(X, \cc{O}_{X})$ be a locally ringed space. If $X$ admits an open cover $\{U_{\alpha}\}_{\alpha \in A}$ where each $(U_{\alpha}, \cc{O}_{X}|_{U_{\alpha}})$ may be identified with $\Spec R_{\alpha}$ for some commutative ring $R_{\alpha}$, we say that $(X, \cc{O}_{X})$ is a \emph{scheme}. 
\end{recollection*}

Although the definitions above are motivated by complex algebraic geometry, the generality of the Zariski spectrum applies to the full gamut of interesting commutative rings which arise in arithmetic and geometry. As before, even if one is only concerned with the study of commutative rings, the category of schemes includes several interesting non-affine moduli to map affine objects to:

\begin{example*}[Projective space]\label{ex:projectivespace}
Given a commutative ring $R$, the following set
  \[
    \frac{\{(\cc{L}, \phi) \mid \cc{L} \text{ an invertible module over }R, ~\phi:\bigoplus_{i=1}^{n}R \twoheadrightarrow \cc{L} \text{ an }R\text{-linear surjection}\}}{\{ R-\text{linear isomorphisms which intertwine the given surjection}\} }
  \]
  is naturally identified with the set of maps of locally ringed spaces from $\Spec R$ to the non-affine scheme $\bb{P}^{n-1}_{\bb{Z}}$.
\end{example*}

To summarize, commutative rings arise as the global sections of the structure sheaves on affine schemes. In this form, geometric methods may be applied towards their study, and such methods will essentially employ non-affine objects.

\subsection{Affine moduli in higher Zariski geometry}

\emph{Tensor-triangular geometry}, after \cite{Balmer2011}, is the study of tensor-triangulated categories (or \emph{tt-categories}) via the use of algebro-geometric methods. This approach, at least philosophically, regards tt-categories as arising from the global sections of ``structure sheaf of tt-categories'' on the affine objects in a tt-analogue of algebraic geometry. This is unfortunately only a heuristic: there are problems with treating structure sheaves valued in tt-categories, as triangulated structure is poorly behaved with respect to forming limits in ordinary categories.

To pursue a geometric program which putatively includes non-affine objects, we must replace tensor-triangulated categories with a better behaved notion. To this end, we work in the $\infty$-category of \emph{$2$-rings}, denoted $\twoCAlg$, defined to be the $\infty$-category of idempotent-complete stably symmetric monoidal $\infty$-categories with symmetric monoidal exact functors. In this setting, the idea that $2$-rings appear as the affine objects of a categorified algebraic geometry can be made precise. This is accomplished in \cite{aokiHigherGeometriesTensor}, utilising the formalism of \emph{geometries} introduced in \cite{lurieDerivedAlgebraicGeometryV}.

\begin{recollection*}Let $\twoCAlg$ denote the $\infty$-category of $2$-rings. There is an $\infty$-category of \emph{locally $2$-ringed spaces}, consisting of pairs $X$ a topological space and $\cc{O} \in \Shv(X; \twoCAlg)$ a sheaf of $2$-rings on $X$, where the sheaf $\cc{O}$ is required to satisfy a certain locality condition\footnote{see \hyperref[ssec:zariskigeomprereq]{Subsection~\ref*{ssec:zariskigeomprereq}} for a full discussion}. There is a functor \[\Spec: \twoCAlg^{\op} \to \{\mm{Locally~2\text{-}ringed~spaces}\}\] known as the \emph{Zariski spectrum}, which is the right adjoint in an adjunction of the following form
    \[
      \Gamma: \{\mm{Locally~2\text{-}ringed~spaces}\} \rightleftarrows \twoCAlg^{\op}: \Spec
    \]
    where $\Gamma$ sends a pair $(X, \cc{O_{X}})$ to the global sections of the sheaf $\cc{O}$ on $X$. The essential image of the functor $\Spec$ is said to comprise the \emph{affine $2$-schemes}.
\end{recollection*}

If schemes in classical algebraic geometry provide interesting non-affine objects which apply towards the study of ordinary rings, it seems only natural that their obvious categorification ought have interesting applications to the study of $2$-rings.

\begin{definition*}
  A locally $2$-ringed space $(X, \cc{O}_{X})$ is a \emph{$2$-scheme} if it admits an open cover $\{U_{\alpha}\}$ such that the locally $2$-ringed spaces $(U_{\alpha}, \cc{O}_{X}|_{U_{\alpha}}) \simeq \Spec \cc{K}_{\alpha}$ for $\cc{K}_{\alpha} \in \twoCAlg$. We say it is a \emph{rigid $2$-scheme} if the $\cc{K}_{\alpha}$ can be selected to be \emph{rigid $2$-rings}, see \autoref{def:rigid}.
\end{definition*}

The main result of this paper provides a characterization of rigid affine $2$-schemes among the collection of all rigid $2$-schemes. Recall that a topological space is \emph{quasiseparated} if the intersection of any two quasicompact open subsets is quasicompact.

\begin{ThmAlpha}\label{thmalph:affineness}
  Let $(X, \cc{O}_{X})$ be a rigid $2$-scheme. Then $(X, \cc{O}_{X})$ is an affine $2$-scheme if and only if the underlying topological space $X$ is quasicompact and quasiseparated (or qcqs).
\end{ThmAlpha}

The result above shows that higher Zariski geometry behaves in stark contrast to classical algebraic geometry, where most non-affine schemes of interest are qcqs, e.g., \hyperref[ex:projectivespace]{projective space} over any base ring and any closed subscheme thereof. We view this as a feature rather than a bug: it demonstrates the category of $2$-rings already contains the representing objects for many moduli problems. In our next result, we utilise our affineness criterion to demonstrate that a particular class of higher Zariski moduli problems arise as affine $2$-schemes. In doing so, we provide a new universal property of the $2$-ring of perfect complexes on a qcqs scheme. 

\begin{recollection*}
  Let $\CAlg$ denote the $\infty$-category of $\bb{E}_{\infty}$-rings, the higher-categorical enhancement of ordinary commutative rings. Recall that there is a limit-preserving functor
  \[
    \cc{R}_{(-)}: \twoCAlg \to \CAlg
  \]
  given by sending a $2$-ring $\cc{K}$ to the endomorphism ring spectrum $\End_{\cc{K}}(\unit)$ of the unit object $\unit \in \cc{K}$. Thus, to any a locally $2$-ringed space $(X, \cc{O})$, one may associate the pair $(X, \cc{R_{O}})$, where $\cc{R_{O}}$ is the sheaf of $\bb{E}_{\infty}$-rings given by
  \[
    \cc{R_{O}}: U \mapsto \cc{R}_{\cc{O}(U)} := \End_{\cc{O}(U)}(\unit).
  \]
  The space $X$ along with this sheaf can be shown to be a \emph{locally spectrally ringed space}: this is the $\infty$-category consisting of pairs $(X, \cc{O})$ of a topological space $X$ equipped with a sheaf $\cc{O}$ of $\bb{E}_{\infty}$ rings on $X$ satisfying a certain locality condition analogous to that of ordinary locally ringed spaces\footnote{see \hyperref[ssec:geometries]{Subsection~\ref*{ssec:geometries}} for a full discussion}. 
\end{recollection*}

The previous recollection yields a restriction functor
  \[
    \{\mm{Locally~2\text{-}ringed~spaces}\} \to \{\mm{Locally~spectrally~ringed~spaces}\}
  \]
given by sending a pair $(X, \cc{O})$ to the pair $(X, \cc{R_{O}})$. This restriction functor will give rise to a natural family of moduli problems:

\begin{definition*}
  Let $X$ be a spectral scheme\footnote{of which ordinary schemes are an example, see \autoref{def:spectralschemes} for the full definition.}. The \emph{relative spectrum} of $X$ is unique locally $2$-ringed space $\Spec^{2}_{1}X$ for which
  \[
    \map((Y, \cc{O}_{Y}), \Spec^{2}_{1}X) \simeq \map((Y, \cc{R}_{\cc{O}_{Y}}), X)
  \]
  where $(Y, \cc{O}_{Y})$ is any arbitrary locally $2$-ringed space. In the above, the left hand mapping space is taken in locally $2$-ringed spaces, and the right-hand mapping space is taken in locally spectrally ringed spaces.
\end{definition*}

The restriction morphism from locally $2$-ringed spaces to locally spectrally ringed spaces may also be obtained via the general framework of \emph{geometries} outlined in \cite{lurieDerivedAlgebraicGeometryV}. In Section 2.1 of \emph{loc.\ cit.} the author constructs the relative spectrum as a left adjoint to restriction, defined on all locally spectrally ringed spaces. This construction takes values in locally $2$-/spectrally ringed $\infty$-\emph{topoi} in lieu of topological spaces, and we will demonstrate that it yields topological spaces for our cases of interest. These notions are recalled in \hyperref[ssec:zariskigeomprereq]{Subsection~\ref*{ssec:zariskigeomprereq}}.

The upshot is that relative spectra always exist, and our next result show that these moduli problems are in fact associated to affine schemes.

\begin{ThmAlpha}\label{thmalph:qcqsschemesareaffine}
  Given a qcqs spectral scheme $X$, its relative spectrum $\Spec^{2}_{1}X$ is a rigid qcqs $2$-scheme. Hence, it is an affine $2$-scheme by \autoref{thmalph:affineness}. The global sections of $\Spec^{2}_{1}X$ may be naturally identified with the $2$-ring $\Perf_{X}$ of perfect complexes on $X$ (\autoref{def:perfectcomplexes}), furnishing an identification
  \[
    \Spec^{2}_{1}X \simeq \Spec \Perf_{X}.
    \]
  These equivalences are natural in the spectral scheme $X$.
\end{ThmAlpha}

We also show that the result above holds in the case of \emph{Dirac} spectral schemes, see \autoref{def:spectralschemes}.

\begin{remark*}
\autoref{thmalph:qcqsschemesareaffine} can be regarded as an instance of the ``1-affineness'' philosophy of \cite{gaitsgorySheavesCategoriesNotion2014} for qcqs spectral schemes. In fact, alongside the the descent results of \cite[\S 5]{aokiHigherGeometriesTensor}, \autoref{thmalph:qcqsschemesareaffine} will recover the fact that such spectral schemes are 1-affine, demonstrated in \cite[\S 5]{gaitsgorySheavesCategoriesNotion2014}. It will also supply the expected extension of these results to the Dirac spectral case.
\end{remark*}

In \cite{aokiHigherGeometriesTensor}, it is shown that the functor $\Spec$ from $2$-rings to locally $2$-ringed spaces restricts to a fully faithful functor on the full subcategory $\twoCAlg_{\rig} \subseteq \twoCAlg$ of rigid $2$-rings. As an immediate corollary, we obtain the following, which is a new universal property for the $2$-ring of perfect complexes on a spectral scheme.

\begin{CorAlpha}\label{coralph:univprop}
  For a rigid $2$-ring $\cc{K}$ and a qcqs spectral or Dirac spectral scheme $X$, there is an identification
  \[
    \map_{\twoCAlg}(\Perf_{X}, \cc{K}) \simeq \map\big((\Spec \cc{K}, \cc{R}_{\cc{O}_{\Spec \cc{K}}}), X\big)
  \]
  where the right-hand mapping space is taken in locally spectrally ringed spaces. Moreover, these equivalences are natural in both the $2$-ring $\cc{K}$ and the spectral scheme $\cc{X}$.
\end{CorAlpha}

As above, we also demonstrate that \autoref{coralph:univprop} holds in the Dirac spectral setting. In the next section we will discuss further consequences.

\begin{remark*}
  For $X$ an affine scheme, \autoref{coralph:univprop} recovers Balmer's method of comparing the spectrum of a given rigid $2$-ring $\cc{K}$ to the Zariski spectrum of the ordinary ring $\pi_{0}\cc{R}_{\cc{K}}$ of homotopy classes of endomorphisms of its unit, introduced in \cite[\S 5]{balmerSpectraSpectraSpectra2010}. Indeed, there is always a canonical inclusion of $2$-rings $\Perf_{\cc{R}_{\cc{K}}} \hookrightarrow \cc{K}$, from which the above corollary furnishes a map  \[
    \Spec \cc{K} \to \Spec \cc{R}_{\cc{K}}
  \] of locally spectrally ringed spaces (or the Dirac analog). This construction is shown to recover his map in \cite[\S 4.D]{aokiHigherGeometriesTensor}, to which we refer the reader for more details.
\end{remark*}

\subsection{Reconstruction of schemes and geometrization of 2-rings} Given a qcqs ordinary scheme $X$, a pioneering result of tensor-triangular geometry due to Thomason \cite{thomasonClassificationTriangulatedSubcategories1997} identifies the underlying space of $\Spec \Perf_{X}$ with the underlying space of $X$, essentially via reduction to work of Hopkins \cite{hopkinsGlobalMethodsHomotopy1987} and Neeman \cite{neemanChromaticTower1992} in the affine noetherian case. In \cite{balmerPresheavesTriangulatedCategories2002} Balmer extends Thomason's method to recover $X$ as a locally ringed space, by working with a tt-categorical implementation of the structure sheaf on $\Spec \Perf_{X}$. 

Using the \autoref{coralph:univprop}, we provide another proof of Thomason and Balmer's results. We proceed by directly constructing a comparison map that applies to arbitrary qcqs spectral or Dirac spectral schemes, which we then use to reduce to the affine noetherian case discussed above. We find the qualitative differences to the their approaches to be of interest, and have elaborated on the same in \autoref{rem:differences1} and \autoref{rem:differences2}.

\begin{ThmAlpha}\label{thmalpha:comparison} Let $X$ be qcqs spectral or Dirac spectral scheme. There is a map of locally spectrally ringed spaces
  \[
    \gamma:  (\Spec \Perf_{X}, \cc{R}_{\cc{O}_{\Spec \Perf_{X}}})\to X
  \]
  such that the structure sheaf of $X$ is identified with $\gamma_{\ast}\cc{R}_{\cc{O}_{\Spec \Perf_{X}}}$. This map and the identification of sheaves can moreover be constructed naturally in $X$. If $X$ is an ordinary qcqs scheme regarded as a spectral scheme, the comparison map $\gamma_{X}$ constructed above is an equivalence.
\end{ThmAlpha}

Note that such a comparison map for \emph{affine} schemes was already defined on the level of ordinary ringed spaces by Balmer, as discussed in the previous section. However, it is nonobvious exactly how to glue these comparison maps to one of the form above naturally in $X$, even just on the level of topological spaces. Here, we utilize the full strength of \autoref{coralph:univprop} to do so.

Finally, we show that \autoref{coralph:univprop} gives a simple method for ``geometrizing'' certain $2$-rings.

\begin{ThmAlpha}\label{thmalph:geometrization}
  Let $\cc{K}$ be a $2$-ring such that the locally spectrally ringed space $(\Spec \cc{K}, \cc{R}_{\cc{O}_{\Spec \cc{K}}})$ is itself a spectral scheme. Then there is a fully faithful embedding
  \[
    \Perf_{(\Spec \cc{K}, \cc{R}_{\cc{O}_{\Spec \cc{K}}})} \hookrightarrow \cc{K}
  \]
  induced by the identity functor on $(\Spec \cc{K}, \cc{R}_{\cc{O}_{\Spec \cc{K}}})$ via the universal property of \autoref{coralph:univprop}. 
\end{ThmAlpha}

\autoref{thmalph:geometrization} is a special case of a slightly stronger statement, which is collected in \autoref{thm:slightlybettergeom}. However, just this case is already sufficient for a broad swath of applications:

\begin{example*}
  Let $G$ be a finite group and $k$ a field of characteristic dividing the order of $G$. The \emph{stable module category} of $kG$, denoted $\mm{St}_{kG}$, is a $2$-ring which exactly captures the failure of the representation theory of $G$ over $k$ to be semisimple (see \autoref{ex:stablemodule} for a precise definition). In the linked example, we show that $(\Spec \mm{St}_{kG}, \cc{R_{O}})$ is itself a spectral scheme whose underlying classical scheme is identified with $\mm{Proj}~H^{\ast}(G, k)$. We refer to this spectral scheme as the \emph{spectral support variety}, and denote it by $\cc{V}_{G}$. \autoref{thmalph:geometrization} now supplies a fully faithful embedding
  \[
    \Perf_{\cc{V}_{G}} \hookrightarrow \mm{St}_{kG}
  \]
  which is moreover functorial in the group $G$.
\end{example*}

Such a map is constructed in \cite[\S 9]{Mathew2016} by hand for the case of $G$ elementary abelian, and in fact Mathew's observation in \emph{loc.\ cit.} formed the genesis for this paper's investigation. In \cite{mathew2015torusactionsstablemodule}, Mathew utilizes this embedding to provide a new proof of Dade's theorem via spectral algebraic geometry. A shadow of the same is recovered by work of \cite{balmerPicardGroupsTriangular2010}, which constructs an injective homomorphism from the Picard group of lines bundles on the variety $\mm{Proj}~H^{\ast}(G, k)$ to the group of invertible objects of $\mm{St}_{kG}$, after inverting the characteristic of the base field, and also by hand. The embedding $\Perf_{\cc{V}_{G}} \hookrightarrow \mm{St}_{kG}$ unifies these two observations, and moreover supplies a functoriality result for Balmer's homomorphism (which, per the author's knowledge, was hitherto unclear). We will expand upon these considerations in \cite{TFEndo}, where we also explore applications to computing the group of \emph{torsion-free endotrivial modules} of $kG$. 

We also demonstrate \autoref{thmalph:geometrization} in the Dirac spectral setting. We have chosen this level of generality owing to recent work of Balmer-Gallauer \cite{balmerGeometryPermutationModules2025} which shows that the Zariski spectrum of the \emph{derived category of permutation modules} of $G$ over $k$ is a Dirac spectral scheme for $G$ an elementary abelian group. We expound upon this example in \autoref{ex:permmods}.

\subsection{Overview} \hyperref[sec:recollections]{Section~\ref*{sec:recollections}} contains the necessary recollections on geometries following \cite{lurieDerivedAlgebraicGeometryV}, in addition to results on the Zariski geometry of $2$-rings following \cite{aokiHigherGeometriesTensor}. In \hyperref[sec:relativeschemes]{Section~\ref*{sec:relativeschemes}} we collect some preliminaries on relative spectra from \cite{lurieDerivedAlgebraicGeometryV} which will enable their computation via descent. Our key result here is \autoref{prop:relativespecislocalic}, which demonstrates that the relative spectra of schemes are $0$-localic. Our first main theorem, \autoref{thmalph:affineness}, is proved in \hyperref[sec:affineness]{Section~\ref*{sec:affineness}}. \hyperref[sec:recon]{Section~\ref*{sec:recon}} is dedicated to the proofs of \autoref{thmalph:qcqsschemesareaffine} and \autoref{thmalpha:comparison}. Finally, \autoref{coralph:univprop} and \autoref{thmalph:geometrization} are recorded in \hyperref[sec:geometrization]{Section~\ref*{sec:geometrization}}. In \hyperref[sec:non-affines]{Appendix~\ref*{sec:non-affines}}, we discuss two naturally occurring classes examples of non-affine $2$-schemes.

\subsection{Conventions} We follow the notational conventions of \cite{aokiHigherGeometriesTensor}. \tdef{This color modifier} is used to mark instances of important definitions or notations for the reader's convenience. Unless otherwise specified, when writing an adjunction between $\infty$-categories as \[L : \cc{C} \rightleftarrows \cc{D} : R,\] it is always understood that $L$ is left adjoint to $R$, or in symbols $L \dashv R$. Per standard convention, we write the symbol $\yo$ to denote the Yoneda embedding, as in the Hiragana \yo\ne\da~(yo-ne-da). 

\subsection{Acknowledgements}

We thank Martin Gallauer for helpful comments on the material of \autoref{ex:permmods}. We are immensely grateful to Lucas Piessevaux and Juan Omar G\'omez for careful readings of our first draft. We thank David Gepner for constant encouragement and ever-present interest in our projects. Finally, we thank Jennifer Cantrell for help with selecting the cover image. Some of this material is based on work conducted while in residence at the Hausdorff Research Institute for Mathematics during the Trimester Program: ``Spectral Methods in Algebra, Geometry, and Topology'' from September--December 2022, funded by the Deutsche Forschungsgemeinschaft under Germany’s Excellence Strategy – EXC-2047/1 – 390685813. 

\section{Recollections on the Geometry of Rings and 2-Rings}\label{sec:recollections}

\subsection{Geometries}\label{ssec:geometries}

In this subsection we recall the classical Zariski and Dirac geometries on $\CAlg$ as constructed in \cite{lurieDerivedAlgebraicGeometry2011a} and \cite[\S 3]{aokiHigherGeometriesTensor}. We then go on to recall the construction of the absolute spectrum and the relative spectrum associated to a morphism of geometries. 

\begin{definition}\label{def:geometry}
    Let $\cc{G}$ be a small $\infty$-category, $\cc{G}^{\mm{ad}}$ a wide subcategory of~$\cc{G}$, and $\tau$ a Grothendieck topology on~$\cc{G}$. We say that $(\cc{G}^{\mm{ad}},\tau)$ is an \tdef{admissibility structure} on~$\cc{G}$ or that $(\cc{G},\cc{G}^{\mm{ad}},\tau)$ is a \tdef{geometry} if the conditions below are satisfied.
    \begin{enumerate}
        \item $\cc{G}$ has finite limits and is idempotent complete.
        \item $\tau$ is generated by morphisms in $\cc{G}^{\mm{ad}}$.
        \item $\cc{G}^{\mm{ad}}$ is closed under base changes in~$\cc{G}$.
        \item If $f$ is a retract of~$g$ in $\cc{G}^{[1]}$ satisfying $g\in\cc{G}^{\mm{ad}}$, then $f\in\cc{G}^{\mm{ad}}$ as well.
    \end{enumerate}
    We refer to morphisms in~$\cc{G}^{\mm{ad}}$ and covers in~$\tau$ as admissible morphisms and admissible covers, respectively.
\end{definition}

      \begin{definition}
    The \tdef{classical Zariski geometry} on $\bb{E}_{\infty}$ ring spectra\footnote{not necessarily connective!} consists of the following data:
        \begin{enumerate}
            \item $\mdef{\GZarsp} =\CAlg^{\omega,\op}$, the opposite of the $\infty$-category of compact objects in $\bb{E}_{\infty}$ rings.
            \item Admissible morphisms correspond to localization maps $R \to R[x^{-1}]$ for $x \in \pi_{0}R$
            \item  A finite collection $\{R\to R[x_i^{-1}]\}_{i \in I}$ generates a covering sieve if the set $\{x_i\}_{i \in I} \subset \pi_{0}R$ generates the unit ideal. 
        \end{enumerate}
      \end{definition}

      \begin{remark}
        There is a variant of the above definition, $\mdef{\GZarspcn}$, which is restriction of the admissibility structure above to the $\infty$-category $\CAlg^{\mm{cn}}$ of connective $\bb{E}_{\infty}$ ring spectra. Note that the inclusion $\CAlg^{\mm{cn}} \hookrightarrow \CAlg$ restricts to compact objects, and as such induces a morphism of geometries $\GZarspcn \to \GZarsp$. We will not consider this variant here.
      \end{remark}

            \begin{definition}
    The \tdef{Dirac geometry} on $\bb{E}_{\infty}$-ring spectra consists of the following data:
        \begin{enumerate}
            \item $\mdef{\GDirsp} =\CAlg^{\omega,\op}$, the opposite of the $\infty$-category of compact objects in $\bb{E}_{\infty}$ rings.
            \item Admissible morphisms correspond to localization maps $R \to R[x^{-1}]$ for homogenous elements $x \in \pi_{\ast}R$
            \item   A finite collection $\{R\to R[x_i^{-1}]\}_{i \in I}$ generates a covering sieve if the set $\{x_i\}_{i \in I} \subset \pi_{\ast}R$ generates the unit ideal. 
        \end{enumerate}
      \end{definition}

      Any geometry $\cc{G}$ gives rise to a category of topological spaces equipped with certain sheaves valued in $\Ind(\cc{G}^{\op})$, generalizing the classical notion of a locally ringed topological space discussed in the introduction. It is generally more convenient to work directly with the category of sheaves on a topological space, which is an example of a \emph{topos}. We will directly work with general topoi in what follows.

      \begin{notation}
        We write $\mdef{\LTop}$ to denote the $\infty$-category of $\infty$-topoi with morphisms given by left-exact left adjoint functors. We write $\mdef{\RTop} \coloneq \LTop^{\op}$.
      \end{notation}

      \begin{definition}
        Given a geometry $\cc{G}$, we write $\mdef{\LTop(\cc{G})}$ to denote the $\infty$-category of \tdef{$\cc{G}$-structured $\infty$-topoi}. This is the $\infty$-category consisting of pairs $(\cc{X}, \cc{O})$ where $\cc{X} \in \LTop$ is an $\infty$-topos and $\cc{O} \in \Shv(\cc{X}; \Ind(\cc{G}^{\op}))$ is an $\Ind(\cc{G}^{\op})$-valued sheaf satisfying a particular locality condition with respect to the admissibility structure on $\cc{G}$. A morphism $(\cc{X}, \cc{O_{X}}) \to (\cc{Y}, \cc{O_{Y}})$ is a pair
        \[ f^{\ast}: \cc{X} \to \cc{Y} \in \LTop^{[1]},\ \phi: f^{\ast}\cc{O_{X}} \to \cc{O_{Y}} \in \Shv(\cc{Y}; \Ind(\cc{G}^{\op}) \]
        where $f^{\ast}$ is used both to refer to left-exact left adjoint in $\LTop^{[1]}$ and its induced pullback functor on $\Ind(\cc{G}^{\op})$-valued sheaves, and $\phi$ is itself required to satisfy a particular locality condition with respect to the admissibility structure on $\cc{G}$.
      \end{definition}
      
      We have opted to omit specifics in the above, and refer to \cite[\S 1.2]{lurieDerivedAlgebraicGeometryV} or \cite[\S 3]{aokiHigherGeometriesTensor} for a full recollection of these definitions. Rather than do this ourselves, let us quickly recall the behaviour of the $\infty$-categories of structured topoi obtained from the geometries above.
      
      \begin{example} The $\infty$-category $\mdef{\mm{LTop}(\GZarsp)}$ of consists of pairs $(\cc{X}, \cc{O}_{\cc{X}})$ where $\cc{X} \in \mm{LTop}$ and $\cc{O}_{\cc{X}} \in \Shv(\cc{X};\CAlg)$ a sheaf of rings which is \emph{local} in the following senses:
          \begin{enumerate}
          \item The sheaf $\cc{O}_{\cc{X}}$ is locally nontrivial, i.e., it is not identically the $0$ ring. 
          \item Let $\cc{O}_{\cc{X}}^{\times}$ denote the sheaf of units of $\cc{O}_{\cc{X}}$, and let $e: \cc{O}_{\cc{X}}^{\times}\to \cc{O}_{\cc{X}}$ denote the canonical inclusion. Then $e \coprod (1-e): \cc{O}_{\cc{X}}^{\times} \coprod \cc{O}_{\cc{X}}^{\times} \to \cc{O}_{\cc{X}}$ is an effective epimorphism.
          \end{enumerate}
          A morphism $(\cc{X}, \cc{O}_{\cc{X}}) \to (\cc{Y}, \cc{O}_{\cc{Y}})$ consists of a left-exact left adjoint $f^{\ast}:\cc{X} \to \cc{Y}$ along with a morphism $f^{\ast}\cc{O}_{\cc{X}} \to \cc{O}_{\cc{Y}} \in \Shv(\cc{O}_{\cc{Y}};\CAlg)$ such that the following square is Cartesian:
          \[\xymatrix{
              f^{\ast}\cc{O}_{\cc{X}}^{\times} \ar[r] \ar[d] & \cc{O}_{\cc{Y}}^{\times} \ar[d] \\
              f^{\ast}\cc{O}_{\cc{X}} \ar[r] & \cc{O}_{\cc{Y}}. 
            } \] 
        \end{example}
        
        \begin{remark}\label{rem:localityonstalks} Given a point $x^{\ast}: \cc{X} \to \An$, it is easy to check that the stalk $\cc{O}_{\cc{X},x} \in \CAlg$ is a \tdef{local $\bb{E}_{\infty}$-ring} in the sense that it is a local ring on $\pi_{0}$, and that any morphism in $\LTop(\GZarsp)$ induces a $\pi_{0}$-local morphism of $\bb{E}_{\infty}$ rings on stalks.
        \end{remark}
        
        \begin{example} The $\infty$-category $\mdef{\mm{LTop}(\GDirsp)}$ consists of pairs $(\cc{X}, \cc{O}_{\cc{X}})$ where $\cc{X} \in \mm{LTop}$ and $\cc{O}_{\cc{X}} \in \Shv(\cc{X};\CAlg)$ a sheaf of rings satisfying the conditions of \cite[Definition 1.2.8]{lurieDerivedAlgebraicGeometryV} for the geometry $\GDirsp$. A morphism $(\cc{X}, \cc{O}_{\cc{X}}) \to (\cc{Y}, \cc{O}_{\cc{Y}})$ consists of a left-exact left adjoint $f^{\ast}:\cc{X} \to \cc{Y}$ along with a morphism $f^{\ast}\cc{O}_{\cc{X}} \to \cc{O}_{\cc{Y}} \in \Shv_{\CAlg}(\cc{O}_{\cc{Y}})$ which is a local transformation of $\GDirsp$-structures in the sense of \emph{loc. cit.}
        \end{example}
        
        \begin{remark}
          Given any point $x^{\ast}: \cc{X} \to \An$ the stalk $\cc{O}_{\cc{X},x} \in \CAlg$ is a \tdef{Dirac-local $\bb{E}_{\infty}$-ring}, i.e., one for which the graded ring obtained from passing to $\pi_{2\ast}$ admits a single homogenous prime ideal. Any morphism in $\LTop(\GDirsp)$ induces a Dirac-local morphism of $\bb{E}_{\infty}$ rings on stalks , i.e., a local map of graded rings upon passage to $\pi_{2\ast}$. 
        \end{remark}

Recall that a transformation of geometries $\cc{G} \to \cc{G}'$ is, informally, a left-exact functor which respects admissible morphisms and admissible covers. The following functor is constructed in \cite[\S 2.1]{lurieDerivedAlgebraicGeometryV}.

\begin{definition}\label{def:relativespectrum} Given a transformation of geometries $\cc{G} \to \cc{G}'$, we write \[\mdef{\Spec^{\cc{G}'}_{\cc{G}}}: \mm{LTop}(\cc{G}) \to \mm{LTop}(\cc{G}')\] to denote the left adjoint to the natural forgetful functor $\mm{LTop}(\cc{G}') \to \mm{LTop}(\cc{G})$. We will call this the \tdef{relative spectrum} functor.
\end{definition}

Given a geometry $\cc{G}$, one always has access to $\mdef{\cc{G}^{\mm{disc}}}$ the \tdef{discrete geometry} on $\cc{G}$, given by the geometry with only equivalences as admissible covers. There is always a transformation of geometries $\cc{G}^{\mm{disc}}\to \cc{G}$, which informally gives rise to a sequence of functors \[\mm{LTop}(\cc{G}) \xrightarrow{\mm{fgt}} \mm{LTop}(\cc{G}^{\mm{disc}}) \xrightarrow{\Gamma(-, \cc{O})}  \Ind(\cc{G}^{\op}).\]

The latter functor $\Gamma(-, \cc{O}) : \mm{LTop}(\cc{G}^{\mm{disc}}) \to  \Ind(\cc{G}^{\op})$ admits a left adjoint, sending $R \in \Ind(\cc{G})$ to the pair $(\An, \underline{R}) \in \mm{LTop}(\cc{G}^{\mm{disc}})$.

\begin{definition}
  We write $\mdef{\Spec^{\cc{G}}}:\Ind(\cc{G}) \to \mm{LTop}(\cc{G})$ to denote the composite \[\Ind(\cc{G}^{\op})\xrightarrow{(\An, \underline{R})} \mm{LTop}(\cc{G}^{\mm{disc}}) \xrightarrow{\Spec^{\cc{G}}_{\cc{G}^{\mm{disc}}}} \mm{LTop}(\cc{G})\] which we refer to as the \tdef{absolute spectrum} functor.
\end{definition}

The following follows from the existence of the adjoints above.

\begin{theorem}\label{thm:absolutespectra}
  One has an adjunction \[\Spec^{\cc{G}} : \Ind(\cc{G}^{\op})  \rightleftarrows  \mm{LTop}(\cc{G}): \Gamma_{\cc{G}} \] where $\Gamma_{\cc{G}}:= \Gamma(-, \cc{O})$. Namely, there is a natural equivalence \[\map_{\mm{LTop}(\cc{G})}(\Spec^{\cc{G}}(-), (\cc{X}, \cc{O}_{\cc{X}})) \simeq \map_{\cc{G}}(-, \Gamma(\cc{X}, \cc{O}_{\cc{X}}))\] of functors from $\Ind(\cc{G}^{\op})$ to $\widehat{\An}$.
\end{theorem}

\begin{example}
  For $R \in \CAlg$, one has an identification $\Spec^{\GZarsp} \simeq (\Shv(\Spec\pi_{0}R), \cc{O}_{\Spec R})$ of locally spectrally ringed topoi. In particular, the identification $\Gamma(\Spec\pi_{0}R, \cc{O}_{\Spec R}) \simeq R$ implies that the absolute spectrum functor is fully faithful. This latter statement implies the \emph{subcanonicity} of the Zariski topology on $\CAlg$, namely that morphisms may be glued along Zariski covers.
\end{example}

\subsection{2-rings} Recall the tensor product on $\Cat^{\mm{perf}}$ of \cite[\S 4.1.2]{ben-zviIntegralTransformsDrinfeld2010a}, informally characterized by the fully faithful inclusion $\Fun^{\mm{ex}}(\cc{C}_1\otimes\dots\otimes\cc{C}_n\to\cc{C}') \hookrightarrow \Fun(\cc{C}_{1} \times \dots \times \cc{C}_{n}, \cc{C}')$ with essential image given by functors which are exact in each variable.

\begin{definition}\label{def:2ring}
    We let $\mdef{\twoCAlg}$ denote the underlying $\infty$-category associated to $\CAlg(\Cat^{\mm{perf}})$. The objects of $\twoCAlg$ are referred to as \tdef{$2$-rings}. 
\end{definition}

Concretely, a $2$-ring is a small symmetric monoidal idempotent-complete stable $\infty$-category $\cc{K} = (\cc{K},\otimes,\unit)$ such that $\otimes$ is exact in both variables, and a morphism between $2$-rings $\cc{K}$ and~$\cc{L}$ is a symmetric monoidal exact functor $f\colon\cc{K}\to\cc{L}$. The homotopy category of a 2-ring naturally has the structure of a small idempotent-complete tt-category. 

\begin{definition}\label{def:rigid}
  The full subcategory $\mdef{\twoCAlg_{\rig}} \subseteq \twoCAlg$ of \tdef{rigid $2$-rings} comprises of those $2$-rings $\cc{K}$ such that any object $x \in \cc{K}$ admits a dual $x^{\vee}$. The subcategory $\twoCAlg_{\rig} \subseteq \twoCAlg$ is coreflective and hence closed under limits, see \cite[Proposition 2.37]{aokiHigherGeometriesTensor}.
\end{definition}

\begin{example}\label{ex:rigidity}
    Given an $\bb{E}_{\infty}$ ring spectrum $R$, the $2$-ring of perfect complexes $\Perf_{R}$ is a rigid $2$-ring.
\end{example}

\begin{definition}
  Given a $2$-ring $\cc{K}$, a \tdef{thick tensor ideal} or \tdef{tt-ideal} of $\cc{K}$ is a stable subcategory $\cc{I} \subseteq \cc{K}$ which is closed under retracts, and is moreover closed under tensoring with any object of $\cc{K}$. Given a subset $\mathscr{G} \subseteq \cc{K}$, we write $\mdef{\langle \mathscr{G}\rangle}$ to denote the smallest thick tt-ideal containing $\mathscr{G}$.
\end{definition}

\begin{definition}
  Given a $2$-ring $\cc{K}$ and a thick tensor-ideal $\cc{I} \subseteq \cc{K}$, the \tdef{Karoubi quotient} of $\cc{K}$ is the initial object of the full subcategory of $\twoCAlg_{\cc{K}/}$ consisting of the symmetric monoidal functors which send every object of $\cc{I}$ to $0$. We denote this object by $\mdef{\cc{K}/\cc{I}}$.
\end{definition}

\begin{remark}
  Karoubi quotients always exist for any tt-ideal, and are the idempotent-complete incarnations of Verdier localizations. Outside of the monoidal setting, these are treated in \cite[Appendix A.3]{calmesHermitianKtheoryStable2025}. For a quick overview of the basic theory in $\twoCAlg$, we refer the reader to \cite[\S 2]{aokiHigherGeometriesTensor}. 
\end{remark}

\subsection{Prerequisite results on the Zariski geometry of 2-rings}\label{ssec:zariskigeomprereq} We will need to utilise the following facts, all of which are imported from \cite[\S 3-4]{aokiHigherGeometriesTensor}.

\begin{definition}[{\cite[Theorem A]{aokiHigherGeometriesTensor}}]
    The following data defines a geometry, known as the \tdef{Zariski geometry}, on commutative 2-rings.
        \begin{enumerate}
            \item $\mdef{\GZarst}:=\twoCAlg^{\omega,\op}$ is the opposite of the $\infty$-category of compact objects in the $\infty$-category of $2$-rings.
            \item A morphism $\cc{K}\to\cc{K}'$ in $(\GZarst)^{\op}$ is called admissible if it corresponds to a Karoubi quotient $\cc{K} \to \cc{K}/\cc{I}$. 
            \item A finite collection of admissible morphisms $\{f_i\colon\cc{K}\to\cc{K}_{i}\}_{i \in I}$ is declared to generate a covering sieve if for every $x \in \bigcap_{i\in I}\ker f_{i}$, there exists an $n$ so that $x^{\otimes n} = 0$.
            \end{enumerate}
\end{definition}

\begin{example}
  The $\infty$-category $\mdef{\mm{LTop}(\GZarst)}$ of consists of pairs $(\cc{X}, \cc{O}_{\cc{X}})$ where $\cc{X} \in \mm{LTop}$ and $\cc{O}_{\cc{X}} \in \Shv(\cc{X};\twoCAlg)$ is a sheaf of 2-rings satisfying the conditions of \cite[Definition 1.2.8]{lurieDerivedAlgebraicGeometryV} for the geometry $\GZarst$. A morphism $(\cc{X}, \cc{O}_{\cc{X}}) \to (\cc{Y}, \cc{O}_{\cc{Y}})$ is as above.
\end{example}

 \begin{remark}
  Given any point $x: \cc{X} \to \An$ the stalk $\cc{O}_{\cc{X},x} \in \twoCAlg$ is a \tdef{local 2-ring} in the sense that if $x \otimes y \in \cc{O}_{\cc{X},x}$ is tensor-nilpotent, then either $x$ or $y$ is tensor-nilpotent. Any morphism $(\cc{X}, \cc{O}_{\cc{Y}}) \to (\cc{Y}, \cc{O}_{\cc{Y}})$ in $\LTop(\GZarst)$ induces nil-conservative morphisms on stalks, namely morphisms with kernels consisting of tensor-nilpotent elements.
 \end{remark}

 As suggested by the terminology, one may compare the classical Zariski geometry of $\bb{E}_{\infty}$ ring spectra and the Zariski geometry of $2$-rings.

\begin{proposition}[{\cite[Proposition 4.33]{aokiHigherGeometriesTensor}}]
    The assignment $A\mapsto\Perf_{A}$ yields morphisms of geometries $\GZarsp \to \GZarst$ and $\GDirsp \to\GZarst$.
\end{proposition} 

From the proposition above, one obtains a series of morphisms $\GZarsp \subseteq \GDirsp \to \GZarst$. The associated restriction functors from $\LTop(\GZarst)$ to $\LTop(\GZarsp)$ (or $\LTop(\GDirsp)$) can be identified through the helpful proposition below.

\begin{definition}\label{def:endfunctor}
     We write $\mdef{\cc{R}_{(-)}}$ to indicate the functor sending $\cc{K} \in \twoCAlg$ to the endomorphism ring spectrum $\hom_{\cc{K}}(\unit, \unit) \in \CAlg$. This functor participates in an adjunction of the form
  \[
    \Perf: \CAlg \rightleftarrows \twoCAlg: \cc{R_{(-)}}
  \]
  where the left adjoint is fully faithful, see for example \cite[Construction 4.29]{aokiHigherGeometriesTensor}.  
\end{definition}

\begin{proposition}[{\cite[Lemma 3.22]{aokiHigherGeometriesTensor}}]
  Given a $\GZarst$-structured $\infty$-topos $(\cc{X}, \cc{O}) \in \LTop(\GZarst)$, its associated restriction to $\LTop(\GZarsp)$ $($or $\LTop(\GDirsp))$ may be identified with the pair $(\cc{X}, \cc{R_{\cc{O}}})$ where $\cc{R_{\cc{O}}}$ is the composite
  \[
    \cc{R}_{(-)} \circ \cc{O}: \cc{X}^{\op}\to \CAlg
    \]
\end{proposition}

In \cite{aokiHigherGeometriesTensor} it is shown that the Zariski geometry of a $2$-ring is captured by its \emph{Balmer spectrum}. We recall this notion and the essential features required below. For a more detailed recollection, we refer either to \cite[Section 4.1]{aokiHigherGeometriesTensor} or to Balmer's original paper \cite{balmerSpectrumPrimeIdeals2005}.

\begin{recollection}
  Let $\cc{K}$ be a $2$-ring. The \tdef{Balmer spectrum} of $\cc{K}$ is the topological space $\Spc \cc{K}$ whose underlying set is given by
  \[
    \{\cc{P} \mid \cc{P} \subseteq \cc{K} \text{ is a prime tt-ideal} \} 
  \]
  Where by \tdef{prime tt-ideal} we mean that $\cc{P}\text{ is a tt-ideal and }\forall x,y \in \cc{K},~x \otimes y \in \cc{P} \implies x \in \cc{P}\text{ or }y \in \cc{P}$. The topology of $\Spc \cc{K}$ is generated by a basis of open subsets $U(a) := \{ \cc{P} \in \Spc \cc{K}\mid a \in \cc{P}\}$. It can be shown that the basic open subsets are exactly the quasicompact open subsets, and with this topology the Balmer spectrum is a quasicompact, quasiseparated spectral space.
\end{recollection}

The following theorem is a combination of \cite[Theorem C, Theorem D]{aokiHigherGeometriesTensor}.

\begin{theorem}\label{thm:zariskiisbalmer}
  Let $\cc{K} \in \twoCAlg$.
  \begin{enumerate}
  \item There is a natural identification of underlying topoi $\Spec^{\GZarst}\cc{K} \simeq \Shv(\Spc \cc{K})$, where the $\Spc \cc{K}$ refers to the Balmer spectrum of the tt-category $\ho \cc{K}$.
  \item    If $\cc{K}$ is moreover assumed to be rigid, then the associated structure sheaf on $\Spec \cc{K}$ may be identified with the unique $\twoCAlg$-valued sheaf on $\Spc \cc{K}$ which sends quasicompact opens of the form $U(a) \subseteq \Spc \cc{K},\ a\in \cc{K}$ to the Karoubi quotients $\cc{K}/\langle a \rangle$.
  \end{enumerate}
\end{theorem}

Note that the global sections of the structure sheaf on $\Spec \cc{K}$ are identified with sections on the open subset $U(0) \subseteq \Spc \cc{K}$. From this, we immediately obtain the following. 

\begin{corollary}\label{thm:zariski2subcanonical}
  Let $\cc{K} \in \twoCAlg_{\rig}$ be a rigid 2-ring. Then the counit of the adjunction $\Spec \dashv \Gamma$ yields an equivalence $\cc{K} \simeq \Gamma(\Spec \cc{K}, \cc{O}_{\Spec \cc{K}})$. In particular, the functor $\Spec: \twoCAlg_{\rig} \to \mm{LTop}(\GZarst)$ is fully faithful.
\end{corollary}

\begin{notation}
  We will henceforth write $\mdef{\Spec \cc{K}}:= \Spec^{\GZarst}\cc{K}$ for $\cc{K} \in \twoCAlg$, and will refer to the underlying space by $|\! \Spec \cc{K}|$. Its associated structure sheaf will be denoted $\mdef{\cc{O}_{\cc{K}}}$.
\end{notation}

\section{Preliminaries on relative spectra and descent}\label{sec:relativeschemes}

In this section we collect certain results of \cite{lurieDerivedAlgebraicGeometryV} which will allow us to compute relative spectra. The key observation here is that the relative spectra of spectral schemes (\autoref{def:spectralschemes}) themselves arise as sheaves on a topological space, or more generally a \emph{frame}. The proof of this claim proceeds via a reduction to affines, where the claim is true by the following example.

\begin{example}\label{ex:relative+of+absolute}
  Let $\cc{G}$ denote either of $\GZarsp$ or $\GDirsp$, and $\cc{G}'$ denote $\GZarst$. Given an $\bb{E}_{\infty}$ ring $R$, the relative spectrum can be computed as
  \[
    \Spec^{\cc{G}'}_{\cc{G}} \Spec^{\cc{G}} R \simeq \Spec^{\cc{G}'} \Perf_{R}
  \]
  essentially by definition of the absolute spectrum. This latter object is identified with sheaves on a topological space by \autoref{thm:zariskiisbalmer}.
\end{example}

The reduction to affine schemes will require us to establish certain \emph{descent} properties of the relative spectrum functor, namely that relative spectra may be computed by patching along open subsets. In the setting of general $\infty$-topoi, the notion of ``open subset of a topological space'' is subsumed by that of an ``\'etale morphism of $\infty$-topoi''. We discuss the same in the next subsection.

\subsection{\'Etale maps} 

Recall the following definition of \cite[\S 6.3.5]{lurieHigherToposTheory2009}. 

\begin{definition}
  A map $f_{\ast}:\cc{X} \to \cc{Y} \in \LTop^{[1]}$ is said to be \tdef{\'etale} if it admits a factorization \[\cc{X} \xrightarrow{\ \pi^{\ast}} \cc{X}_{/U} \simarrow  \cc{Y}\] where $\pi^{\ast}$ is right adjoint to the projection $\pi_{!}: \cc{X}_{/U} \to \cc{X}$ for some $U \in \cc{X}$.
\end{definition}

Note that under these conditions, $\pi^{\ast}$ is itself in $\LTop^{[1]}$.

\begin{definition}\label{def:etale}
  Let $\cc{G}$ be a geometry. A morphism $f: (\cc{X}, \cc{O}_{\cc{X}}) \to (\cc{Y}, \cc{O}_{\cc{Y}}) \in \LTop(\cc{G})^{[1]}$ is said to be \tdef{\'etale} if the following conditions are satisfied:
  \begin{enumerate}
  \item The underlying geometric morphism $f^{\ast}: \cc{X} \to \cc{Y}$ is \'etale.
  \item The induced morphism $f^{\ast}\cc{O}_{\cc{X}} \to \cc{O}_{\cc{Y}}$ is an equivalence in $\Shv(\cc{Y}; \Ind(\cc{G}^{\op}))$.
  \end{enumerate}
\end{definition}

\begin{notation}
  We write $\mdef{\LTop_{\et}} \subset \LTop$ and $\mdef{\LTop(\cc{G})_{\et}} \subset \LTop(\cc{G})$ to indicate the wide subcategories spanned by the \'etale morphisms. Note that the induced forgetful functor $\LTop(\cc{G})_{\et} \to \LTop_{\et}$ is a left fibration.
\end{notation}

\'Etale morphisms of $\infty$-topoi with $\cc{G}$-structure behave analogously to open immersions in classical algebraic geometry. They also turn out to behave well after applying the relative spectrum functor. We collect several of these features in the following proposition, the proofs of which are recorded in \cite[Proposition 2.3.5, Proposition 2.3.18]{lurieDerivedAlgebraicGeometryV}. 

\begin{proposition}\label{prop:omnibusetale} Let $\cc{G}$ be a geometry.
  \begin{enumerate}
  \item\label{prop:omnibusetale1} $\LTop(\cc{G})_{\et}$ admits small limits which are preserved by the inclusion $\LTop(\cc{G})_{\et} \to \LTop(\cc{G})$. Furthermore, a small diagram in $\LTop(\cc{G})_{\et}$ is a limit diagram if and only if the diagram of underlying topoi is a limit diagram in $\LTop_{\et}$, and hence also in $\LTop$. 
  \item\label{prop:omnibusetale2} For every $\cc{X} \in \LTop(\cc{G})$, one has equivalences $\cc{X}^{\op} \simeq (\LTop(\cc{G})_{\et})_{/(\cc{X}',\cc{O}')}$ via $U \mapsto (\cc{X}_{/U}, \cc{O}|_{U})$. 
  \item\label{prop:omnibusetale3} Given a morphism of geometries $f: \cc{G} \to \cc{G}'$ and an \'etale morphism \[(\cc{X}, \cc{O_{X}})\to (\cc{X}_{/U}, \cc{O_{X}}|_{U}) \in \LTop(\cc{G})_{\et}^{[1]}\] the associated map $\Spec^{\cc{G}'}_{\cc{G}}(\cc{X}, \cc{O}) \to \Spec^{\cc{G}'}_{\cc{G}}(\cc{X}_{/U}, \cc{O_{X}}|_{U}) \in \LTop(\cc{G}')^{[1]}$ is also \'etale, and moreover sits in the following co-Cartesian diagram 
  \[\xymatrix{
  (\cc{X}, \cc{O_{X}})\ar[r] \ar[d] & (\cc{X}_{/U}, \cc{O_{X}}|_{U}) \ar[d] \\
\Spec^{\cc{G}'}_{\cc{G}}(\cc{X}, \cc{O}) \ar[r] & \Spec^{\cc{G}'}_{\cc{G}}(\cc{X}_{/U}, \cc{O_{X}}|_{U}) 
  }
  \]
  in $\LTop$, where the vertical morphisms are associated to the unit transformations of the adjunction $\Spec^{\cc{G}'}_{\cc{G}} \dashv \mm{res}$.
  \end{enumerate}
\end{proposition}

For the following lemma, fix a morphism of geometries $f: \cc{G} \to \cc{G}'$, and let $(\cc{X}, \cc{O}) \in \LTop(\cc{G})$ be a fixed base.

\begin{lemma}\label{lem:relativespecsheaf}
 The relative spectrum functor $\Spec^{\cc{G}'}_{\cc{G}}$ sends limits in $\LTop(\cc{G})_{\et}$ to limits in $\LTop(\cc{G}')$.
\end{lemma}

\begin{proof}
  We first demonstrate the simpler statement that the relative spectrum sends limits in $(\LTop(\cc{G})_{\et})_{(\cc{X},\cc{O})/}$ to limits in $\LTop(\cc{G}')$, given $(\cc{X}, \cc{O}) \in \LTop(\cc{G})$. Write $(\cc{X}', \cc{O}'):= \Spec^{\cc{G}'}_{\cc{G}}(\cc{X}, \cc{O}) \in \LTop(\cc{G}')$ and let $\eta^{\ast}: \cc{X} \to \cc{X}' \in \LTop^{[1]}$ denote the induced counit map. \autoref{prop:omnibusetale} implies that the relative spectrum construction lifts to a functor $(\LTop(\cc{G})_{\et})_{(\cc{X},\cc{O})/} \to (\LTop(\cc{G}')_{\et})_{(\cc{X}',\cc{O}')/}$, and the same proposition implies limits in this latter $\infty$-category may be computed in $\LTop(\cc{G}')$. It thus suffices to show that this lift preserves limits. Applying the equivalences of \hyperref[prop:omnibusetale2]{\autoref*{prop:omnibusetale}.(2)} to $(\cc{X}, \cc{O})$ and $(\cc{X}', \cc{O}')$ yields a composite of the following form
  \[\cc{X}^{\op} \simarrow (\LTop(\cc{G})_{\et})_{(\cc{X},\cc{O})/} \rightarrow (\LTop(\cc{G}')_{\et})_{(\cc{X}',\cc{O}')/} \simarrow \cc{X}^{\prime\hspace{1pt}\op}\]
  which may be identified with the opposite of $\eta^{\ast}: \cc{X} \to \cc{X}'$. As $\eta^{\ast}$ preserves colimits, the claim follows.

  We next observe that for any $(\cc{X}, \cc{O}) \in \LTop(\cc{G})_{\et}$, the forgetful functor $(\LTop(\cc{G})_{\et})_{(\cc{X}, \cc{O})/} \to \LTop(\cc{G})_{\et}$ creates limits. To see this, note that the resultant composite
  \[
    \cc{X}^{\op} \simarrow (\LTop(\cc{G})_{\et})_{(\cc{X}, \cc{O})/} \to \LTop(\cc{G})_{\et} \to \LTop_{\et}
  \] creates limits by \cite[Lemma 6.3.5.21]{lurieHigherToposTheory2009}, and the functor $\LTop(\cc{G})_{\et} \to \LTop_{\et}$ creates limits by \hyperref[prop:omnibusetale1]{\autoref*{prop:omnibusetale}.(1)}. The claim also holds for any geometry in place of $\cc{G}$.

  We now prove the general statement. Let $p: K \to \LTop(\cc{G})_{\et}$ be an arbitrary small diagram, and let $p^{\triangleleft}: K^{\triangleleft} \to \LTop(\cc{G})_{\et}$ be its associated limit diagram with limit $\varprojlim p \coloneq p^{\triangleleft}(\infty)$. We will show that the image of $p^{\triangleleft}$ in $\LTop(\cc{G}')$ remains a limit diagram. Consider the following commutative square
  \[
    \xymatrix{
      (\LTop(\cc{G})_{\et})_{\varprojlim p/} \ar[rr]^{\Spec^{\cc{G}'}_{\cc{G}}} \ar[d] & & (\LTop(\cc{G}')_{\et})_{\Spec^{\cc{G}'}_{\cc{G}}\! (\varprojlim p)/} \ar[d] \\
      \LTop(\cc{G})_{\et} \ar[rr]^{\Spec^{\cc{G}'}_{\cc{G}}} & & \LTop(\cc{G}')_{\et}
    }
  \]
  whose top vertical arrow preserves limits by the first paragraph, and whose vertical arrows create limits by the second paragraph. Using this, we learn that the unique lift of $p^{\triangleleft}$ to $(\LTop(\cc{G})_{\et})_{\varprojlim p/}$ must itself have been a limit diagram, and furthermore that its image under the upper composite must have also been a limit diagram. By the square above, the image of this lift upon applying the upper composite is exactly $\Spec^{\cc{G}'}_{\cc{G}} \circ p^{\triangleleft}$ and hence $\Spec^{\cc{G}'}_{\cc{G}}$ preserved this limit. The first part of \hyperref[prop:omnibusetale1]{\autoref*{prop:omnibusetale}.(1)} now implies that $\LTop(\cc{G}')_{\et} \to \LTop(\cc{G}')$ preserved this limit as well.
\end{proof}

We conclude with the following definition.

\begin{definition}\label{def:canonicaltopology}
  Given a geometry $\cc{G}$, let $\mdef{(\LTop(\cc{G})^{\op}, \can)}$ denote the (locally large, very large) site where $\mdef{\can}$ is the Grothendieck topology consisting of exactly those sieves $\cc{C}_{/(\cc{X}, \cc{O})} \subset (\LTop(\cc{G})^{\op})_{/(\cc{X}, \cc{O})}$ containing a family of maps $f_{i}: (\cc{X}_{i}, \cc{O}_{i}) \to (\cc{X}, \cc{O})$ satisfying:
\begin{enumerate}
\item Each $f_{i}$ is \'etale.
\item Under the equivalence $(\LTop(\cc{G})^{\op}_{\et})_{/(\cc{X}, \cc{O})} \simeq \cc{X}$ of \hyperref[prop:omnibusetale2]{\autoref*{prop:omnibusetale}.(2)}, $f_{i}$ corresponds to a family $\{f_{i}: U_{i} \to \unit\}$ with $\coprod_{I}U_{i} \twoheadrightarrow \unit$ an effective epimorphism.
\end{enumerate}
\end{definition}

\begin{example}
  The Yoneda embedding $\yo: \LTop(\cc{G})^{\op} \to \Fun(\LTop(\cc{G}), \widehat{\An})$ has essential image contained in $\Shv_{\can}(\LTop(\cc{G})^{\op}; \widehat{\An})$. 
\end{example}

\begin{example}\label{ex:relativespecstructuresheaf}
    Given a morphism of geometries $f:\cc{G} \to \cc{G}'$, the assignment \[\{(\cc{X}, \cc{O}) \mapsto \Gamma_{\cc{G}'}(\Spec^{\cc{G}'}_{\cc{G}}(\cc{X}, \cc{O})\} \in \Fun(\LTop(\cc{G}), \Ind(\cc{G}^{\prime \op}))\] is a sheaf on $(\LTop(\cc{G}), \can)$. Indeed, let $\cc{C}_{/(\cc{X}, \cc{O})}$ be an arbitrary covering sieve. By definition, $\cc{C}_{/(\cc{X}, \cc{O})} \hookrightarrow \cc{X}_{/1}$ is associated to a covering sieve of $\unit \in \cc{X}$ in the canonical topology, and thus $ (\cc{X}, \cc{O_{X}}) \simeq \varprojlim_{\cc{C}_{/(\cc{X}, \cc{O})}}(U, \cc{O}|_{U})$ in $(\LTop(\cc{G})_{\et})_{(\cc{X}, \cc{O})/}$. It follows that the natural map \[\Gamma_{\cc{G}'}(\Spec^{\cc{G}'}_{\cc{G}}(\cc{X}, \cc{O}_{\cc{X}})) \simarrow {\varprojlim}_{\cc{C}_{/(\cc{X}, \cc{O})}}\Gamma_{\cc{G}'}(\Spec^{\cc{G}'}_{\cc{G}}(U, \cc{O}|_{U}))\] is an equivalence by the proof of \autoref{lem:relativespecsheaf} and the fact that $\Gamma_{\cc{G}'}$ is a right adjoint, yielding the claim.
\end{example}

\subsection{Localic structured topoi}

Recall that we write $\RTop$ to denote $\LTop^{\op}$. The following definition isolates the class of $\infty$-topoi which actually arise as sheaves on a space, or more generally sheaves on a frame: for example, the Zariski spectra of $2$-rings, discussed in \autoref{ex:zariskispectralocalic}.

\begin{definition}
  Let $\cc{X} \in \RTop$. We say $\cc{X}$ is \tdef{$0$-localic} if for any $\cc{Y} \in \RTop$ one has an equivalence
  \[
    \Fun^{\mm{geom}}(\cc{Y}, \cc{X}) \simeq \Fun^{\mm{geom}}(\tau_{\leq -1}\cc{Y}, \tau_{\leq -1}\cc{X})
  \]
  where $\tau_{\leq -1}\colon\RTop \to \Cat$ sends any $\infty$-topos to the full subcategory of its $(-1)$-truncated objects.
\end{definition}

Recall that a \tdef{frame} is a partially ordered set admitting arbitrary colimits and finite limits, such that finite limits distribute over infinite colimits. We write $\mdef{\mm{Loc}}$ to denote the category of frames with morphisms given by right adjoints which admit left-exact left adjoints. Given a frame $F$, we furthermore write $\Shv(F):= \Fun^{\mm{lim}}(F^{\op}, \An)$. The following result is an agglomeration of the results of \cite[6.4.2.1, 6.4.5]{lurieHigherToposTheory2009}.

\begin{theorem}[Omnibus 0-localic topoi]\label{thm:zerolocalic}
  The functor $\Shv(-) :  \mm{Loc} \to \RTop$ is fully faithful with essential image exactly the $0$-localic topoi. It furthermore admits a left adjoint, given by $\cc{X} \mapsto \tau_{\leq -1}\cc{X}$, whose unit transformation is referred to as the \mdef{$0$-localic reflection}.
\end{theorem}

Above, we write $\tau_{\leq -1}\cc{X} \subseteq \cc{X}$ to refer to the subcategory of \emph{$(-1)$-truncated objects}, those objects $U \in \cc{X}$ whose mapping spaces $\map(-,U)$ are either empty or contractible. It is clear from this that $\tau_{\leq -1} \cc{X}$ is actually a partial order, and it can be shown to be a frame.

\begin{notation} We choose to work in the following settings, mirroring algebro-geometric convention.  
  \begin{enumerate}
  \item Let $\mdef{\RTop^{\mm{loc}}_{\CAlg}} := \LTop(\GZarsp)^{\op}$. We refer to this as the $\infty$-category of \tdef{locally spectrally ringed topoi}.
  \item Let $\mdef{\RTop^{\Dir}_{\CAlg}} := \LTop(\GDirsp)^{\op}$. We refer to this as the $\infty$-category of \tdef{Dirac-locally spectrally ringed topoi}.
  \item Let $\mdef{\RTop^{\mm{loc}}_{\twoCAlg}} := \LTop(\GZarst)^{\op}$. We refer to this as the $\infty$-category of \tdef{locally 2-ringed topoi}.
  \end{enumerate}
  In each of the cases above, morphisms $(\cc{X}, \cc{O}_{\cc{X}}) \to (\cc{Y}, \cc{O}_{\cc{Y}})$ are \emph{geometric} morphisms $f_{\ast}: \cc{X} \to \cc{Y}$ along with morphisms $\cc{O}_{\cc{Y}} \to f_{\ast}\cc{O}_{\cc{X}}$ satisfying a locality condition on their mates. 
\end{notation}

\begin{example}
  There is a full subcategory $\mdef{\mm{Top}^{\mm{loc}}_{\CAlg}} \hookrightarrow \mm{LTop}(\GZarsp)^{\op}$ corresponding exactly to the \emph{locally spectrally-ringed spaces} with hypercomplete sheaves of commutative rings. These have a more familiar definition purely in terms of the pointwise condition of \autoref{rem:localityonstalks}. We refer the reader to \cite[\S 2]{lurieDerivedAlgebraicGeometry2011a} for details.
\end{example}

\begin{notation}
  We will henceforth write $\mdef{\mm{Loc}_{\twoCAlg}} := \RTop_{\twoCAlg} \times_{\RTop} \mm{Loc}$, and analogously with $\mdef{\mm{Loc}^{\loc}_{\twoCAlg}}$.
\end{notation}

\begin{example}\label{ex:zariskispectralocalic}
  Given $\cc{K} \in \twoCAlg$, \autoref{thm:zariskiisbalmer} implies that the underlying $\infty$-topos of $\Spec \cc{K}$ is 0-localic and hence $\Spec \cc{K} \in \mm{Loc}_{\twoCAlg}^{\loc}$.
\end{example}

\subsection{Relative spectra of spectral schemes are 0-localic} We recall one equivalent definition of the $\infty$-category of spectral schemes, following \cite[Definition 2.7]{lurieDerivedAlgebraicGeometry2011a}.

\begin{definition}\label{def:spectralschemes}
  A \tdef{spectral scheme}\footnote{these are potentially nonconnective by default.}, (resp. \tdef{Dirac spectral scheme}) is an object $X \in \RTop^{\loc}_{\CAlg}$ (resp. $\RTop^{\Dir}_{\CAlg}$ satisfying the following two conditions.
  \begin{enumerate}
  \item The underlying $\infty$-topos (which we will denote $\mdef{\Shv(X)}$) is 0-localic. 
  \item There is an effective epimorphism $\{\coprod U_{i} \twoheadrightarrow \unit\}$ in $\Shv(X)$ such that for every $i$ there exists $R\in \CAlg$ and an equivalence $(\Shv(X)_{/U_{i}}, \cc{O}_{X}|_{U_{i}}) \simeq \Spec R$ in $\mm{LTop}(\GZarsp)$ (resp. $\mm{LTop(\GDirsp)}$).
  \end{enumerate}
 We write $\mdef{\SpSch} \subset \RTop^{\loc}_{\CAlg}$, (resp. $\mdef{\SpSch^{\Dir}}$) to denote the full subcategory of spectral schemes. 
\end{definition}

\begin{definition}
    We say a spectral scheme is \tdef{quasicompact and quasiseparated (qcqs)} if the underlying 0-localic $\infty$-topos $\Shv(X)$ is coherent in the sense of \cite[\S 3]{lurieDerivedAlgebraicGeometry2011a}.
\end{definition}

The following result is the main import of this subsection. 

\begin{proposition}\label{prop:relativespecislocalic}
  Let $\cc{G} = \GZarsp$ \emph{(}resp. $\GDirsp$\emph{)}, let $\cc{G}'= \GZarst$. Let $X \in \SpSch, \SpSch^{\Dir}$. Then $\Spec^{\cc{G}'}_{\cc{G}}X \in \mm{Loc}^{\loc}_{\twoCAlg}$.
\end{proposition}

We will need the following establishing lemmas.

\begin{lemma}\label{lem:lexrespectstrunc}
Any left exact functor $F: \cc{C} \to \cc{D}$ between $\infty$-categories admitting finite products\footnote{This condition is easily dropped by passing to presheaf categories.} sends $(-1)$-truncated objects of $\cc{C}$ to $(-1)$-truncated objects of $\cc{D}$. 
\end{lemma}
\begin{proof}
  This boils down to the claim that an object $c \in \cc{C}$ is $(-1)$-truncated if and only if the diagonal map $\Delta: c \to c \times c$  is an equivalence, the verification of which we leave to the reader.
\end{proof}

\begin{lemma}\label{lem:localicdevissage}
  Let $f_{\ast}:\cc{X} \to \cc{Y} \in \RTop^{[1]}$ be a geometric morphism satisfying the following conditions:
  \begin{enumerate}
  \item $\cc{Y}$ is 0-localic.
  \item There is an effective epimorphism $\coprod_{I} U_{i} \twoheadrightarrow \unit \in \cc{Y}^{[1]}$ such that $U_{i} \in \tau_{\leq -1}\cc{Y}$ and $\cc{X}_{/f^{\ast}U_{i}}$ is $0$-localic, for every $i\in I$.
  \end{enumerate}
  Then $\cc{X}$ is $0$-localic.
\end{lemma}
\begin{proof}
  Passing to associated left adjoints supplies a factorization of $f^{\ast}: \cc{Y} \to \cc{X}$ through a left-exact left adjoint $\cc{Y} \to \Shv(\tau_{\leq -1}\cc{X})$ using \autoref{thm:zerolocalic}. Using \autoref{lem:lexrespectstrunc} we may now reduce to the case where $\cc{Y} \simeq \Shv(\tau_{\leq -1}\cc{X})$ and $f_{\ast}$ is the reflection $\eta_{\ast}:\cc{X} \to \Shv(\tau_{\leq -1}\cc{X})$. In this case, we are given $\{U_{i}\}_{i \in I} \subset \tau_{\leq -1}\Shv(\tau_{\leq -1}\cc{X}) \simeq \tau_{\leq -1}\cc{X}$ along with an effective epimorphism $\coprod_{I}U_{i} \twoheadrightarrow \unit$ in $\cc{X}$. Define the morphisms $f^{i}_{\ast} \in \RTop^{[1]}$ as in the following Cartesian square
  \begin{equation}\label{eq:localicdevissage}\xymatrix{
      \cc{X}_{/U_{i}} \ar[r] \ar[d]_{f_{\ast}^{i}} & \cc{X} \ar[d]^{\eta_{\ast}} \\
      \Shv(\tau_{\leq -1}\cc{X})_{/U_{i}} \ar[r] & \Shv(\tau_{\leq -1} \cc{X}) \\
    }
  \end{equation}
  noting that the object $f^{i, \ast}U_{i} \in \Shv(\tau_{-1}\cc{X})$ may be identified with $U_{i}$. 
  
  We first claim that the left adjoint to the top horizontal map in \eqref{eq:localicdevissage}, given by \[U_{i} \times - : \cc{X} \to \cc{X}_{/U_{i}},\] induces an equivalence $(\tau_{-1}\cc{X})_{/U_{i}} \simeq \tau_{-1}(\cc{X}_{/U_{i}})$. Since $U_{i}$ is $(-1)$-truncated, we have that $x \times_{U_{i}} y \simeq x \times y$ for any pair of maps $x \to U_{i}$, $y\to U_{i}$ in $\cc{X}^{[1]}$, owing to the following Cartesian diagram 
\[
\xymatrix{
x \times_{U_{i}} x \ar[r] \ar[d] & x \times x \ar[d]\\
U_{i} \ar@{=}[r]& U_{i} \times U_{i}.
}
\]
We find that every $x \to U_{i} \in \cc{X}^{[1]}$ satisfies $x \simeq x \times_{U_{i}} U_{i} \simeq x \times U_{i}$, and that any $x \to U_{i}$ is $(-1)$-truncated if and only if $x \in \tau_{-1}\cc{X}$ from which the claim follows. 

The previous claim furnishes an equivalence $\Shv(\tau_{\leq -1} (\cc{X}_{/U_{i}})) \to \Shv(\tau_{\leq -1}\cc{X})_{/ U_{i}}$ which moreover identifies $f^{i}_{\ast}$ with the 0-localic reflection $\cc{X}_{/U_{i}} \to \Shv(\tau_{\leq -1}(\cc{X}_{/U_{i}}))$ (by checking that this induces the identity upon application of $\tau_{\leq -1}$). By assumption on $U_{i}$, we have that $f^{i}_{\ast}$ is an equivalence for every $i \in I$, whence we have that $\eta_{\ast}: \cc{X} \to \Shv(\tau_{\leq -1}\cc{X})$ is an equivalence upon pulling back along $\Shv(\tau_{\leq -1}\cc{X})_{/\coprod_{I} U_{i}} \to \Shv(\tau_{\leq -1}\cc{X})$. Applying descent \cite[Theorem 6.1.3.9]{lurieHigherToposTheory2009}, we conclude.
\end{proof}

Before we begin the proof of \autoref{prop:relativespecislocalic}, we will also require the following key fact.

\begin{recollection}\label{ex:relativeaffinespectra}
Given a transformation of geometries $\cc{G} \to \cc{G}'$, \cite[Proposition 2.3.18.(2)]{lurieDerivedAlgebraicGeometryV} supplies an identification $\Spec^{\cc{G}'}_{\cc{G}}\Spec^{\cc{G}} \simeq \Spec^{\cc{G}'}$ of functors $\Ind(\cc{G}^{\op}) \to \mm{LTop}(\cc{G}')$.
\end{recollection}

\begin{proof}[Proof of \autoref{prop:relativespecislocalic}] By construction, we are supplied with a counit map $f_{\ast} : \Spec^{\cc{G}'}_{\cc{G}}X \to X$ in $\RTop^{\loc}_{\CAlg}$. \cite[Theorem 2.40]{lurieDerivedAlgebraicGeometry2011a} implies that we may find an effective epimorphism $\coprod_{i \in I} U_{i} \twoheadrightarrow X$ where $U_{i} \in \tau_{\leq -1}\Shv(X)$ and $(\Shv(X)_{/U_{i}}, \cc{O}_{X}|_{U_{i}}) \simeq \Spec A_{i}$ for $A_{i}\in \CAlg$; indeed, we may select any cover of the underlying classical scheme by affine opens. \hyperref[prop:omnibusetale3]{\autoref*{prop:omnibusetale}.(3)} now implies that \[(\Spec^{\cc{G}'}_{\cc{G}}X)_{/f^{\ast}U_{i}} \simeq \Spec^{\cc{G}'}_{\cc{G}}\Spec A_{i} \simeq \Spec \Perf_{A_{i}}\] the last of which is 0-localic by \autoref{ex:zariskispectralocalic}. We are now in the setting of \autoref{lem:localicdevissage}, and we may conclude.
\end{proof}
\section{An Affineness Criterion for 2-Schemes}\label{sec:affineness}

This section is dedicated to the proof of \autoref{thmalph:affineness}, recorded as \autoref{thm:affineness} below. Let us first define our basic objects.

\begin{definition}
    An \tdef{affine $2$-scheme} is an object $(\cc{X}, \cc{O_{X}}) \in \RTop^{\loc}_{\twoCAlg}$ which is equivalent to $\Spec \cc{K}$ for some $\cc{K} \in \twoCAlg$. We write $\mdef{\twoAff} \subseteq \RTop^{\loc}_{\twoCAlg}$ to denote the full subcategory of affine $2$-schemes.  
\end{definition}

\begin{definition}\label{def:twoscheme}
  A \tdef{2-scheme} is an object $(\cc{X}, \cc{O}_{\cc{X}}) \in \RTop^{\loc}_{\twoCAlg}$ satisfying the following conditions:
  \begin{enumerate}
  \item The underlying $\infty$-topos $\cc{X}$ is 0-localic.
  \item There is an effective epimorphism $\{\coprod U_{i} \twoheadrightarrow \unit\}$ in $\cc{X}$ such that:
  \begin{enumerate}
      \item[(a)] Each $U_{i} \in \tau_{\leq -1} \cc{X}$
      \item[(b)] For every $i$ there exists $\cc{K}_{i}\in \twoCAlg$ and an equivalence $(\cc{X}_{/U_{i}}, \cc{O_{X}}|_{U_{i}}) \simeq \Spec \cc{K}_{i}$ in $\mm{LTop}(\GZarst)$.
  \end{enumerate}
  \end{enumerate}
 We write $\mdef{\twoSch} \subset \RTop^{\loc}_{\twoCAlg}$ to denote the full subcategory of 2-schemes.
\end{definition}

The truncatedness restriction on $U_{i}$ in part (a) of condition (2) above is meant to ensure that the covering $2$-affine neighborhoods of $\cc{X}$ are actually associated to subframes of $\tau_{\leq -1} \cc{X}$. This at first appears to distinguish it from \autoref{def:spectralschemes}. However, condition (2) turns out to be redundant, and we refer the reader to \autoref{rem:whytruncated} at the end of this subsection for further discussion of the same.

\begin{definition}
     We say a 2-scheme is \tdef{quasicompact and quasiseparated (qcqs)} if the underlying 0-localic $\infty$-topos $\cc{X}$ is coherent in the sense of \cite[\S 3]{lurieDerivedAlgebraicGeometry2011a}.
\end{definition}

\begin{example}\label{ex:balmerspectrumiscoherent}
  By the identification of \autoref{thm:zariskiisbalmer} and the results of \cite[\S 2]{balmerSpectrumPrimeIdeals2005}, one has that $\Spec \cc{K} \in \twoSch_{\qcqs}$ for every $\cc{K} \in \twoCAlg$, and hence $\twoAff \subseteq \twoSch_{\mm{qcqs}}$.
\end{example}

\begin{definition}
     We say a $2$-scheme $(\cc{X}, \cc{O}_{\cc{X}})$ is \tdef{rigid} if there exists an effective epimorphism $\{\coprod U_{i} \twoheadrightarrow \unit\}$ in $\cc{X}$ such that for every $i$:
     \begin{enumerate}
         \item $U_{i} \in \tau_{\leq -1} \cc{X}$
         \item There exists $\cc{K}_{i} \in \twoCAlg_{\rig}$ and an equivalence $(\cc{X}_{/U_{i}}, \cc{O_{X}}|_{U_{i}}) \simeq \Spec \cc{K}_{i}$ in $\mm{LTop}(\GZarst)$.
     \end{enumerate}
\end{definition}

\begin{remark}
    As one might hope, it is possible to show that the result of \autoref{thm:affineness} implies that $(\cc{X}, \cc{O_X})$ is a rigid $2$-scheme if and only if \emph{any} \'etale map \[\Spec \cc{K} \to (\cc{X}, \cc{O_X}) \in (\RTop^{\loc}_{\CAlg})^{[1]}\] associated to $U \in \tau_{\leq -1} \cc{X}$ factors through a map $f:\Spec \cc{K} \to \Spec \cc{K}'$ where $\cc{K}'$ is rigid and $f$ is an equivalence of locally $2$-ringed topoi. Thus, the property of an affine neighborhood to be equivalent to the spectrum of a rigid $2$-ring satisfies the ``affine commmunication lemma''. 
\end{remark}

The main result of this section is the following.

\begin{theorem}\label{thm:affineness}
  A rigid 2-scheme $(\cc{X}, \cc{O}_{\cc{X}})$ is affine if and only if it is qcqs. 
\end{theorem}

The ``only if'' direction is just \autoref{ex:balmerspectrumiscoherent}. For the reverse direction, we will need to collect some preliminaries on the behaviour of Karoubi quotients in $\Catperf$. We are primarily after \autoref{cor:verdierbasechange}, which is a symmetric monoidal and idempotent-complete version of \cite[Proposition A.1.18]{calmesHermitianKtheoryStable2025}. In the interest of self-containment we have chosen to include a direct proof that does not assume their result. We will first need the following useful lemma, a weak version of the ``second isomorphism theorem'' for stable $\infty$-categories.

\begin{lemma}\label{lem:secondiso}
  Let $\cc{E} \subseteq \cc{C} \subseteq \cc{D}$ be a sequence of inclusions in $\Catperf$. Then the induced functor $\cc{C}/\cc{E} \to \cc{D}/\cc{E}$ is fully faithful.
\end{lemma}
\begin{proof}
  Given $x_{1}, x_{2} \in \cc{C}$, \cite[Theorem I.3.3]{nikolausTopologicalCyclicHomology2018a} implies that the map $\map_{\cc{C}/\cc{E}}(e(x_{1}), e(x_{2})) \to \map_{\cc{D}/\cc{E}}(f(x_{1}), f(x_{2}))$ may be identified with the natural map \[\varprojlim_{Z \in \cc{E}\downarrow_{\cc{C}}x_{2}}\map_{\cc{C}}(x_{1}, \cofib(Z \to x_{2})) \to \varprojlim_{Z \in \cc{E}\downarrow_{\cc{D}}x_{2}}\map_{\cc{C}}(x_{1}, \cofib(Z \to x_{2}))\] where $e:\cc{C} \to \cc{C}/\cc{E}$ and $f:\cc{D} \to \cc{D}/\cc{E}$ are the associated Karoubi quotients. However, since $\cc{E} \subseteq \cc{C} \subseteq \cc{D}$ are inclusions of full subcategories, the induced functor $\cc{E}\downarrow_{\cc{C}}x_{2} \to \cc{E}\downarrow_{\cc{D}}x_{2}$ is an equivalence, and for every $Z \in \cc{E} \downarrow_{\cc{D}}x_{2}$ the following map is an equivalence \[\map_{\cc{C}}(x_{1}, \cofib(Z \to x_{2})) \simarrow \map_{\cc{D}}(x_{1}, \cofib(Z\to x_{2})).\] The claim follows. 
\end{proof}

\begin{recollection}\label{rec:vectimes}
  We recall the ``oriented fiber product'' construction of \cite{landKtheoryPullbacks2019}. Given a diagram $\cc{A} \xrightarrow{p} \cc{C} \xleftarrow{q} \cc{B}$ in $\Catperf$, the oriented fiber product $\cc{A} \vectimes_{\cc{C}}\cc{B}$ is defined by the following pullback diagram 
  \[ \xymatrix{
      \cc{A} \vectimes_{\cc{C}}\cc{B} \ar[r] \ar[d] & \cc{C}^{[1]} \ar[d]^{s \times t} \\
      \cc{A} \times \cc{B} \ar[r]^{p \times q} & \cc{C}\times \cc{C}
    }
  \]
  where $s, t$ send a morphism in $\cc{C}^{[1]}$ to its source and target respectively. Note that there is a fully faithful inclusion $\cc{A} \times_{\cc{C}} \cc{B} \hookrightarrow \cc{A} \vectimes_{\cc{C}} \cc{B}$, with essential image exactly those triples $(x, y, \alpha: p(x) \to q(y))$ such that $\alpha$ is an equivalence in $\cc{C}$, using the identification $\cc{C} \simeq \cc{C}^{[1], \mm{eq}}$, where the latter refers to the full subcategory on the equivalences, and the fact that the following is a pullback diagram
  \[\xymatrix{
      \cc{A} \times_{\cc{C}} \cc{B} \ar[r] \ar[d] & \cc{C} \simeq \cc{C}^{[1], \mm{eq}} \ar[d] \\
      \cc{A} \vectimes_{\cc{C}}\cc{B} \ar[r] & \cc{C}^{[1]} \\
    }
    \]
\end{recollection}

For the next lemma, let $\cc{K}_{1} \xrightarrow{p} \cc{K}_{12} \xleftarrow{q}\cc{K}_{2}$ be a diagram in $\twoCAlg$ such that $q$ is identified with a Karoubi quotient of $\cc{K}_{2}$ away from $\cc{I}_{2} \subset \cc{K}_{2}$.

\begin{lemma}\label{lem:orientedbasechange}
  The induced functor $\cc{K}_{1}\vectimes_{\cc{K}_{12}}\cc{K}_{2} \to \cc{K}_{1}\vectimes_{\cc{K}_{12}}\cc{K}_{12}$ given by \[(x, y, \alpha: p(x) \to q(y)) \mapsto (x, q(y), \alpha: p(x) \to q(y))\] is a Karoubi quotient with kernel given by the thick subcategory $(0, \cc{I}_{2}, 0)$.
\end{lemma}
\begin{proof}
  Let us write $\cc{C}:= \cc{K}_{1} \vectimes_{\cc{K}_{12}}\cc{K}_{2}$, $\cc{C}' := \cc{K}_{1} \vectimes_{\cc{K}_{12}}\cc{K}_{12}$, and $\cc{D} := (0, \cc{I}_{2}, 0)$. The identification of the kernel of the map $\cc{C} \to \cc{C}'$ as in the lemma is clear, and one thus has a conservative functor $\cc{C}/\cc{D} \to \cc{C}'$ which we first claim is fully faithful.  By \cite[Theorem I.3.3]{nikolausTopologicalCyclicHomology2018a}, it is equivalent to show that given $(x_{1}, y_{1}, \alpha), (x_{2},y_{2},\beta) \in \cc{C}$ arbitrary, the natural map 
  \[
    \varinjlim_{Z\in \cc{D}_{/(x_{2},y_{2},\beta)}}\map_{\cc{C}}\Big((x_{1}, y_{1}, \alpha), \cofib(Z \to (x_{2},y_{2},\beta))\Big) \to \map_{\cc{C}'}((x_{1}, q(y_{1}), \alpha), (x_{2},q(y_{2}),\beta)) 
    \]
    is an equivalence. Using the pullback presentation of $\cc{C}'$, the latter mapping space may be expressed via the following Cartesian square
    \[\xymatrix{
        \map_{\cc{C}'}((x_{1}, q(y_{1}), \alpha), (x_{2},q(y_{2}),\beta)) \ar[r] \ar[d] & \map_{\cc{K}_{12}^{[1]}}(\alpha, \beta) \ar[d] \\
        \map_{\cc{K}_{1}}(x_{1},x_{2}) \times \map_{\cc{K}_{12}}(q(y_{1}),q(y_{2}))\ar[r] & \map_{\cc{K}_{12}}(p(x_{1}),p(x_{2})) \times \map_{\cc{K}_{12}}(q(y_{1}),q(y_{2}))
      }
    \]
    where we note that the following natural map is an equivalence \[\map_{\cc{K}_{1}}(x_{1},x_{2}) \times \varinjlim_{Z \in {\cc{I}_{2}}_{/y_{2}}}\map_{\cc{K}_{12}}(y_{1},\cofib(Z \to y_{2})) \to \map_{\cc{K}_{1}}(x_{1},x_{2}) \times \map_{\cc{K}_{12}}(q(y_{1}),q(y_{2}))\] since $\cc{K}_{2} \to \cc{K}_{12}$ identifies the target with the Karoubi quotient away from $\cc{I}_{2}$. Using the fact that the target map $t: \cc{C} \to \cc{K}_{2}$ is a colocalization with left adjoint given by $y \mapsto (0,y,0)$, we have that the target map induces an equivalence $\cc{D}_{/ (x_{1},y_{2},\beta)} \simeq {\cc{I}_{2}}_{/y_{2}}$. Altogether, we obtain a Cartesian diagram
    \[\xymatrix{
        \varinjlim_{Z\in \cc{D}_{/(x_{2},y_{2},\beta)}}\map_{\cc{C}}\Big((x_{1}, y_{1}, \alpha), \cofib(Z \to (x_{2},y_{2},\beta))\Big) \ar[d] \ar[r] & \map_{\cc{C}'}((x_{1}, q(y_{1}), \alpha), (x_{2},q(y_{2}),\beta)) \ar[d] \\
        \map_{\cc{K}_{1}}(x_{1},x_{2}) \times \varinjlim_{Z \in {\cc{I}_{2}}_{/y_{2}}}\map_{\cc{K}_{12}}(y_{1},\cofib(Z \to y_{2})) \ar[r] & \map_{\cc{K}_{1}}(x_{1},x_{2}) \times \map_{\cc{K}_{12}}(q(y_{1}),q(y_{2}))
        }
      \]
      by taking a filtered colimit of the associated Cartesian squares over $Z \in {\cc{I}_{2}}_{/y_{2}} \simeq \cc{D}_{(x_{2},y_{2},\beta)}$. Since the bottom horizontal arrow is an equivalence, we deduce the claimed fully faithfulness. 
      
      It remains to see that the map $\cc{C} \to \cc{C}'$ has a retract-dense essential image, or equivalently that $\cc{C}/\cc{D} \hookrightarrow \cc{C}'$ is surjective. For this, we note that given any arbitrary $(x,y, \alpha) \in \cc{C}'$, the object $(x \oplus \Sigma x, y \oplus \Sigma y, \alpha \oplus \Sigma \alpha)$ sits in the following cofiber sequence:
      \[
        (0, y \oplus \Sigma y, 0) \rightarrow (x \oplus \Sigma x, y \oplus \Sigma y, \alpha \oplus \Sigma \alpha) \rightarrow (x \oplus \Sigma x, 0, 0)
      \]
      where the third term lives in the essential image of the embedding $\cc{K}_{1} \hookrightarrow \cc{C}'$, and the first term lives in the essential image of the composite $\cc{K}_{2} \to \cc{K}_{2}/\cc{I}_{2} \simeq \cc{K}_{12} \hookrightarrow \cc{C}'$, using the K-theory extension theorem of Neeman-Thomason \cite[Corollary 0.9]{neemanConnectionTheoryLocalization1992}. Since both of these embeddings factor through the map $\cc{C} \to \cc{C}'$, it follows that the middle term of the above sequence is in the full subcategory $\cc{C}/\cc{D} \subseteq \cc{C}'$. It follows that $(x,y, \alpha)$ is the retract of an object in $\cc{C}/\cc{D}$, yielding the claim.
  \end{proof}

  \begin{corollary}\label{cor:verdierbasechange}
    Notation as in the previous lemma, one has that the natural map $\cc{K}_{1} \times_{\cc{K}_{12}}\cc{K}_{2} \to \cc{K}_{1}$ is a Karoubi quotient away from the ideal $(0,\cc{I}_{2}, 0)$. Moreover, the following diagram is co-Cartesian in $\twoCAlg$:
    \[\xymatrix{
        \cc{K}_{1} \times_{\cc{K}_{12}} \cc{K}_{2} \ar[r] \ar[d] & \cc{K}_{1} \ar[d] \\
        \cc{K}_{2} \ar[r] & \cc{K}_{12} \\
      }
      \]
    \end{corollary}
    \begin{proof}
      It is clear that the following composite \[\cc{K}_{1} \times_{\cc{K}_{12}}\cc{K}_{2} \hookrightarrow \cc{K}_{1} \vectimes_{\cc{K}_{12}}\cc{K}_{2} \to \cc{K}_{1} \vectimes_{\cc{K}_{12}} \cc{K}_{12}\] has essential image contained in the full subcategory of triples $(x,y, \alpha)$ such that $\alpha \in \cc{K}_{12}^{[1],\mm{eq}}$ which by \autoref{rec:vectimes} corresponds exactly to the image of $\cc{K}_{1} \times_{\cc{K}_{12}} \cc{K}_{12} \subseteq \cc{K}_{1} \vectimes_{\cc{K}_{12}}\cc{K}_{12}$. Applying \autoref{lem:orientedbasechange} and \autoref{lem:secondiso}, one obtains a fully faithful inclusion $(\cc{K}_{1} \times_{\cc{K}_{12}}\cc{K}_{2})/(0,\cc{I}_{2},0) \hookrightarrow \cc{K}_{1}$, and it remains to show that this is essentially surjective. This uses the same argument as in the last paragraph of \autoref{lem:orientedbasechange}. 

      The fact that the square is co-Cartesian follows from \cite[Corollary 2.30]{aokiHigherGeometriesTensor}, which states that the following square is co-Cartesian in $\twoCAlg$:
       \[\xymatrix{
        \cc{K}_{1} \times_{\cc{K}_{12}} \cc{K}_{2} \ar[r] \ar[d] & \cc{K}_{1} \ar[d] \\
        \cc{K}_{2} \ar[r] &  \cc{K}_{2}/(0,\cc{I}_{2},0)\\
      }
      \]
      where we have also written $(0,\cc{I}_{2},0)$ to refer to its image in $\cc{K}_{2}$. However, this is exactly $\cc{I}_{2} \subseteq \cc{K}_{2}$, which identifies the bottom horizontal arrow above with the localization map $\cc{K}_{2} \to \cc{K}_{12}$.
    \end{proof}

  \begin{proof}[Proof of \autoref{thm:affineness}]
    We will induct on the number of affines required to cover a rigid qcqs 2-scheme. The case $n = 1$ is trivial. For the inductive case, let $(\cc{X}, \cc{O_{X}})$ admit a cover by affine 2-subschemes $U_{1},...,U_{n}$, where each $U_{i} \simeq \Spec \cc{K}_{i}$ for $\cc{K}_{i} \in \twoCAlg_{\rig}$. Let $U := |(\coprod_{i=1}^{n-1}U_{i})^{\times \bullet}| \in \cc{X}$, and consider the scheme $(\cc{X}_{/U}, \cc{O_{X}}|_{U})$\footnote{spatially, this corresponds to the union of the open subschemes $U_{1},...,U_{n-1}$.}. By the inductive hypothesis, $(\cc{X}_{/U}, \cc{O_{X}|_{U}}) \simeq \Spec \cc{L}$ where $\cc{L} := \cc{O}(U) \in \twoCAlg_{\rig}$. By the assumptions of coherence of $(\cc{X}, \cc{O}_{\cc{X}})$ and \autoref{thm:zariskiisbalmer}, one has that the object $U \times U_{n} \in \tau_{\leq -1}\cc{X}_{/U_{n}} \simeq \tau_{\leq -1}\Spec \cc{K}_{n}$ corresponds to a quasicompact open of the Balmer spectrum. By the same theorem and rigidity, the map $(\cc{X}_{/U \times U_{n}}, \cc{O}_{X}|_{U \times U_{n}}) \hookrightarrow (\cc{X}_{/U_{n}}, \cc{O}_{X}|_{U_{n}})$ may be canonically identified with the map $\Spec \cc{K}_{n}/\langle a \rangle \to \Spec \cc{K}_{n}$ associated to the Karoubi quotient of $\cc{K}_{n}$ by some $a \in \cc{K}_{n}$. Similarly, one may identify the map $(\cc{X}_{/U \times U_{n}}, \cc{O}_{\cc{X}}|_{U \times U_{n}}) \hookrightarrow (\cc{X}_{/U}, \cc{O}_{X}|_{U})$ with the localization map associated to some $b \in \cc{L}$. Using that $U \coprod_{U \times U_{n}} U_{n} \simeq \unit \in \cc{X}$ and \autoref{prop:omnibusetale}, we have the following co-Cartesian square in $\RTop^{\loc}_{\twoCAlg}$:
    \begin{equation}\label{eq:affine1}\xymatrix{
        \Spec \cc{L}/\langle b \rangle \simeq \Spec \cc{K}_{n}/\langle a \rangle  \ar[r] \ar[d]  & \Spec \cc{K}_{n} \ar[d] \\
        \Spec \cc{L} \ar[r] & (\cc{X}, \cc{O_{X}})
      }
    \end{equation}
    where all morphisms are \'etale. Passing to global sections and applying \autoref{cor:verdierbasechange}, one obtains the following Cartesian diagram
    \begin{equation}\label{eq:affine2}\xymatrix{
        \Spec \cc{K}_{n}/\langle a \rangle \simeq \Spec \cc{L}/\langle b \rangle \ar[r] \ar[d] & \Spec \cc{K}_{n} \ar[d] \\
        \Spec \cc{L} \ar[r] & \Spec \Gamma(\cc{X}, \cc{O_{X}}) 
      }
    \end{equation}
    where the vertical and horizontal morphisms are each associated to Karoubi quotients of $\Spec \Gamma(\cc{X}, \cc{O}_{\cc{X}})$, thus every morphism is \'etale.

    Under the equivalence of categories $(\RTop^{\mm{loc}}_{\twoCAlg, \et})_{/\Spec \Gamma(\cc{X},\cc{O}_{\cc{X}})} \simeq \Spec \Gamma(\cc{X}, \cc{O}_{\cc{X}})$ of \hyperref[prop:omnibusetale2]{\autoref*{prop:omnibusetale}.(2)}, the \'etale map $\Spec \cc{L} \times \Spec \cc{K}_{n} \to \Spec \Gamma(\cc{X}, \cc{O}_{\cc{X}})$ corresponds to an effective epimorphism $U \coprod V \twoheadrightarrow \unit$ and hence an identification $\unit \simeq U \coprod_{U \times V} V$. Unwinding the equivalences and applying \hyperref[prop:omnibusetale1]{\autoref*{prop:omnibusetale}.(1)}, we find that the Cartesian diagram of \eqref{eq:affine2} is in fact co-Cartesian as well. There is an ``affinization map'' from the square of \eqref{eq:affine1} to that of \eqref{eq:affine2} induced by the unit of the adjunction $\Gamma \dashv \Spec$. As this is the identity on every vertex except the bottom right, we find that the affinization map $(\cc{X}, \cc{O_{X}}) \to \Spec \Gamma(\cc{X}, \cc{O_{X}})$ is an equivalence, yielding the inductive case and hence the claim. 
  \end{proof}

As promised, we elaborate on the apparent difference between \autoref{def:spectralschemes} and \autoref{def:twoscheme} before concluding.

\begin{remark}\label{rem:whytruncated}
    In \autoref{def:twoscheme} the restriction to truncated objects in part (a) of condition (2) is actually unnecessary, as we now explain.
    
    First we claim that even with just condition (1) and part (b) of condition (2), the underlying frame $\tau_{\leq -1}\cc{X}$ can be shown to be \emph{spatial}, i.e., it is the frame of open subsets of a topological space. Write $\mdef{\mm{Frm}} \simeq \mm{Loc}^{\op}$. Using part (b) of condition (2), the assumption that $\{\coprod U_{i} \to \unit\}$ is an effective epimorphism, we obtain the following equalizer diagram in $\Frm$: \[\tau_{\leq -1}\cc{X} \to \prod \tau_{\leq -1}(\cc{X}_{/U_{i}}) \rightrightarrows \prod \tau_{\leq -1}(\cc{X}_{/U_{i}\times U_{j}}).\] Recall that $\tau_{\leq -1}\cc{X}_{/U_{i}} $ is associated to the frame of opens of the topological space $|\! \Spec \cc{K}_{i}| \simeq \Spc \cc{K}_{i}$ by \autoref{thm:zariskiisbalmer}. The claim now follows from the fact that Stone Duality implies spatial frames are closed under limits (\cite[\S II.3]{Johnstone1986-th}, see also \cite[Recollection 4.8]{aokiHigherGeometriesTensor}).

 Keeping the conditions as above, recall that the maps $\Spec \cc{K}_{i} \to \cc{X} \in \RTop^{[1]}$ are \'etale maps of $\infty$-topoi. In fact, by \cite[Lemma 2.3.16]{lurieDerivedAlgebraicGeometryV} the objects $U_{i}$ are in $\tau_{\leq 0}\cc{X}$ since $\cc{X}$ is $0$-localic and $\cc{X}_{/U_{i}} \simeq \Spec \cc{K}_{i}$ is $0$-localic. Restricting to underlying ordinary topoi yields a map 
 \[
\Shv(|\! \Spec \cc{K}_{i}|; \mm{Sets}) \simeq \Shv(|\tau_{\leq -1}\cc{X}|; \mm{Sets})_{/U_{i}} \to \Shv(|\tau_{\leq -1}\cc{X}|; \mm{Sets})
 \] from which it follows that underlying morphism of spaces $\phi: |\! \Spec \cc{K}_{i}| \to |\tau_{\leq -1} \cc{X}|$ is a local homeomorphism, associated to the \emph{espace \'etal\'e} of the sheaf of sets $U_{i}$ on $|\tau_{\leq -1}\cc{X}|$. 
 
 Given any point $x \in |\tau_{\leq -1} \cc{X}|$ in the image of $\phi$, $x$ is contained in an open neighborhood $U$ which is homeomorphic via the map above to an open neighborhood $U(a) \subseteq |\! \Spec \cc{K}_{i}|$, as these form a basis for the Balmer spectrum. Moreover, since this is an \'etale map of \emph{locally $2$-ringed topoi}, we may deduce that $\phi$ induces equivalences
    \[(\cc{X}_{/U}, \cc{O}|_{U}) \simeq (\Shv(U), \cc{O}|_{U}) \simeq U(a) \simeq \Spec \cc{K}_{i}/\langle a \rangle\] of locally $2$-ringed topoi, where by construction $U \in \tau_{\leq -1}\cc{X}$. Selecting such neighborhoods around every point $x \in |\tau_{\leq -1} \cc{X}|$ gives rise to a collection of objects in $\tau_{\leq -1}\cc{X}$ satisfying the desired parts (a) and (b) of condition (2) in \autoref{def:twoscheme}.
    
\end{remark}

\section{Affineness and Reconstruction of Schemes}\label{sec:recon}

\subsection{Relative affineness of schemes} In this section we demonstrate that the relative spectra of qcqs (Dirac) spectral schemes are affine $2$-schemes. To identify their global sections, we will need to introduce the $2$-ring of perfect complexes on a spectral or Dirac spectral scheme. We follow the construction given in \cite[Construction 6.2.1.7]{lurieSpectralAlgebraicGeometry}.

\begin{definition}\label{def:perfectcomplexes}
  As per our convention, we write $\mdef{\widehat{\twoCAlg}}$ to denote $\CAlg(\widehat{\Cat}^{\perf})$, the very large $\infty$-category of large 2-rings. We define the functor of \tdef{perfect complexes on presheaves} 
  \[
    \mdef{\Perf} : \Fun(\CAlg, \widehat{\An})^{\op} \to \widehat{\twoCAlg}
    \]
  as the unique limit-preserving extension of the functor $\Perf : \CAlg \to \twoCAlg \subseteq \widehat{\twoCAlg}$.
\end{definition}

\begin{lemma}\label{lem:perfissmall}
  Consider the functor-of-points map
  \[\SpSch \subseteq \RTop^{\loc}_{\CAlg} \xrightarrow{\yo} \Fun(\RTop^{\loc,\op}_{\CAlg}, \widehat{\An}) \xrightarrow{-\circ \Spec} \Fun(\CAlg, \widehat{\An}).\]
  The functor $\Perf$ of \autoref{def:perfectcomplexes} takes values in $\twoCAlg_{\rig} \subseteq \widehat{\twoCAlg}$ when evaluated against any object in the essential image of $\SpSch$. The same holds if $\SpSch$ and $\RTop^{\loc}_{\CAlg}$ are replaced by $\SpSch^{\Dir}$ and $\RTop^{\loc}_{\Dir}$.
\end{lemma}
\begin{proof}
  \cite[Theorem 2.4.1]{lurieDerivedAlgebraicGeometryV} implies that the essential image of the functor-of-points factors through $\Fun(\CAlg, \An) \subseteq \Fun(\CAlg, \widehat{\An})$. In particular, any object $X \in \SpSch$ has essential image given by a small colimit of the form in $\Fun(\CAlg, \widehat{\An})$. In particular, as $\Perf$ is limit-preserving, the value of $\Perf$ on $X$ may be presented as a small limit in $\twoCAlg$ of objects of the form $\Perf_{R}$ for $R \in \CAlg$. As these latter objects are small and rigid, and the inclusion $\twoCAlg_{\rig} \subseteq \widehat{\twoCAlg}$ is closed under small limits, the result follows. The same argument holds in the Dirac case.
\end{proof}

Given any $X \in \SpSch, \SpSch^{\Dir}$ we have just shown that the rigid $2$-ring $\Perf_{X}$ is small, and hence one may apply the relative spectrum to it. With this observation in tow, we are ready to prove the main result of this section, \autoref{thmalph:qcqsschemesareaffine}.

\begin{theorem}\label{thm:affinenessofschemes}
  Let $\cc{G} = \GZarsp$ \emph{(}resp. $\GDirsp$\emph{)} and let $\cc{G}' = \GZarst$. Then for any $X \in \SpSch_{\qcqs}$ \emph{(}resp. $\SpSch_{\qcqs}^{\Dir}$\emph{)} one has a natural identification
  \[
    \Spec^{\cc{G}'}_{\cc{G}}X \simeq \Spec \Perf_{X}
    \]
\end{theorem}

\begin{proof}
  \autoref{prop:relativespecislocalic} and \hyperref[prop:omnibusetale2]{\autoref*{prop:omnibusetale}.(2)} imply that $\Spec^{\cc{G}'}_{\cc{G}}X$ is a qcqs $2$-scheme. To see this, consider any collection of objects $U_{i} \in \tau_{\leq -1}\Shv(X)$ which jointly cover $\unit \in \Shv(X)$ and satisfy $(\Shv(X)_/{U_{i}}, \cc{O}_{X} |_{U_{i}}) \simeq \Spec R_{i}$. \hyperref[prop:omnibusetale3]{\autoref*{prop:omnibusetale}.(3)} sends this to an \'etale cover of $\Spec^{\cc{G}'}_{\cc{G}}X$ by objects of the form $\Spec \Perf_{R_{i}}$, each of which is associated to a $(-1)$-truncated object of the underlying $\infty$-topos. As these are rigid by \autoref{ex:rigidity}, the relative spectrum is a rigid qcqs $2$-scheme. \autoref{thm:affineness} now implies that the relative spectrum is affine, and the full result follows from the identification \[\Gamma(\Spec^{\cc{G}'}_{\cc{G}}X, \cc{O}) \simeq \Perf_{X},\] which is done below in \autoref{lem:sheafifyingperf}.
\end{proof}

We demonstrate the missing proposition (\autoref{lem:sheafifyingperf}) required to complete the proof of \autoref{thm:affineness}.

\begin{construction}\label{cons:compmap} Consider the following diagram of left adjoints:
  \begin{equation}\label{eq:predescentsquare}\xymatrix{
 \CAlg \ar[rr]^{\Spec} \ar[d]_{\Perf}& &\LTop(\cc{G}) \ar[d]^{\Spec^{\cc{G}'}_{\cc{G}}} \ar@{}[dll]|\nwtrans \ar@{}[dll]<-1.4ex>|\alpha \\
   \twoCAlg \ar[rr]_{\Spec} & & \LTop(\cc{G}')
}
 \end{equation}
where $\alpha \in \Fun(\CAlg, \LTop(\cc{G}'))^{[1]}$ is the natural equivalence of \autoref{ex:relativeaffinespectra}; note that we have dropped the superscripts from the horizontal functors and leave the geometries implicit. Using the adjunction of \autoref{thm:absolutespectra} and passing to horizontal mates, we obtain a comparison morphism:
\begin{equation}\label{eq:unsheafybalmermap}
  \alpha^{h}: \Perf_{\Gamma(-,\cc{O})}  \Rightarrow \Gamma(\Spec^{\cc{G}'}_{\cc{G}}(-), \cc{O}) \in \Fun(\LTop(\cc{G}), \twoCAlg)^{[1]}
\end{equation}
 
\end{construction}

\begin{definition}
  We refer to the admissible topologies on $\CAlg^{\op}$ under the identifications $\CAlg^{\op} \simeq \Pro(\GZarsp)$, $\CAlg^{\op} \simeq \Pro(\GDirsp)$ as the \tdef{Zariski} and \tdef{Dirac} topologies respectively. Here the notion of admissible is as in \cite[Notation 2.2.2]{lurieDerivedAlgebraicGeometryV}: For a given geometry $\cc{G}$, recall that an \tdef{admissible} morphism $f: U \to X$ for $U, X \in \Pro(\cc{G})$ is one for which there exists a pullback diagram as follows
\[\xymatrix{
    U \ar[r]^{f} \ar[d] & U' \ar[d]^{f'} \\
    X \ar[r] & X' \\
  }
  \]
where $U', X' \in \cc{G}$ and $f'$ is an admissible morphism for $\cc{G}$. 
\end{definition}

The next proposition will explicitly identify $\alpha^{h}$ and its target using Zariski (resp. Dirac) descent, thereby concluding this subsection.

\begin{proposition}\label{lem:sheafifyingperf}
  The inclusion \[\Shv_{\can}(\SpSch;\widehat{\twoCAlg}) \subseteq \Fun(\SpSch^{\op}, \widehat{\twoCAlg})\] admits a left adjoint localization. We refer to both this localization and its unit transformation as \mdef{sheafification with respect to the canonical topology}. Keeping the notation of \autoref{cons:compmap}, we have:
  \begin{enumerate}
  \item The map $\alpha^{h} \in \Fun(\SpSch^{\op}, \twoCAlg)^{[1]}$ is identified with the sheafification of the source with respect to the canonical topology.
  \item The sheafification of $\Perf_{\Gamma(-,\cc{O})}$ with respect to the canonical topology on $\SpSch$ is given by the assignment $X \mapsto \Perf_{X}$.
  \end{enumerate}
  The same results hold if $\SpSch$ is replaced by $\SpSch^{\Dir}$.
\end{proposition}
\begin{proof} Our first task is to construct the desired left adjoint localization. For our purposes, we will need to derive an explicit formula in terms of Zariski sheaves on $\CAlg^{\op}$. Let $\widehat{\Shv}_{\mm{Zar}}(\CAlg^{\op})$ denote the $\infty$-category of Zariski sheaves on $\CAlg^{\op}$ valued in $\widehat{\An}$. Recall that the inclusion \[\widehat{\Shv}_{\mm{Zar}}(\CAlg^{\op}) \subseteq \Fun(\CAlg, \widehat{\An})\] admits a left-exact left adjoint localization by applying \cite[Lemma 6.2.2.7]{lurieHigherToposTheory2009} in a larger universe. The functor of points embedding \cite[Theorem 2.4.1]{lurieDerivedAlgebraicGeometryV} states that the following composite is fully faithful \[\SpSch \xrightarrow{\yo} \Fun(\SpSch, \widehat{\An}) \xrightarrow{- \circ \Spec} \Fun(\CAlg, \widehat{\An}) \to \widehat{\Shv}_{\mm{Zar}}(\CAlg^{\op})\] and furthermore sends colimits in $\SpSch_{\et} \subseteq \LTop(\GZarsp)_{\et}$ to colimits in $\Shv_{\mm{Zar}}(\CAlg)$. Thus, for any very large $\infty$-category $\cc{C}$ admitting all large colimits, the induced right adjoint functor $\Shv_{\mm{Zar}}(\CAlg^{\op}; \cc{C}) \to \Fun(\SpSch, \cc{C})$ lifts to $\Shv_{\can}(\SpSch; \cc{C})$ using the same argument as in \autoref{ex:relativespecstructuresheaf}. This gives rise to the following adjunction
  \[
    \iota^{\ast}: \Shv_{\can}(\SpSch;\cc{C}) \rightleftarrows \Shv_{\mm{Zar}}(\CAlg^{\op};\cc{C}) : \iota_{\ast}
  \]
  here,  $\iota^{\ast}$ is given by restricting a sheaf along the map of sites $\Spec: (\CAlg^{\op}, \mm{Zar}) \hookrightarrow (\SpSch, \can)$. To describe $\iota_{\ast}$, note that there is an equivalence \[\Shv_{\mm{Zar}}(\CAlg^{\op}; \cc{C}) \simeq \Fun^{\mm{R}}(\widehat{\Shv}_{\mm{Zar}}(\CAlg^{\op})^{\op}, \cc{C})\] which is inverse to the restriction from the right hand side to the left hand side, see \cite[Proposition 1.3.1.7]{lurieSpectralAlgebraicGeometry}. Under this equivalence, $\iota_{\ast}$ sends a sheaf $\mathscr{F} \in \Shv_{\mm{Zar}}(\CAlg^{\op}; \cc{C})$ to the composite \[\SpSch^{\op} \xrightarrow{\yo} \widehat{\Shv}_{\can}(\SpSch)^{\op} \to \widehat{\Shv}_{\mm{Zar}}(\CAlg^{\op})^{\op} \xrightarrow{\mathscr{F}} \cc{C},\] see for example \cite[Lemma 5.2.6.7]{lurieHigherToposTheory2009}. 

  Now recall that by the subcanonicity of the Zariski topology, the functor \[\Spec: \CAlg^{\op} \to \SpSch \to \Fun(\CAlg, \widehat{\An})^{\op}\] may be identified with the Yoneda embedding. From this, it is easy to directly compute that the counit $\iota^{\ast}\iota_{\ast} \simarrow \mm{id}$ as endofunctors of $\widehat{\Shv}_{\mm{Zar}}(\CAlg^{\op})$. Additionally, $\iota^{\ast}$ is conservative, owing to the following two facts:
  \begin{itemize}
  \item Every quasi-affine spectral scheme $X$ admits a canonical cover by affines $\{\coprod \Spec R_{i} \twoheadrightarrow X\}$ with all terms in the associated \v{C}ech complex affine.
  \item Every object $X \in \SpSch$ admits a canonical cover by affines with all terms in the associated \v{C}ech complex quasi-affine.
  \end{itemize}
  It follows that the restriction $\iota^{\ast}$ is conservative, as any morphism $\mathscr{F} \to \mathscr{G} \in \Shv_{\can}(\SpSch)^{[1]}$ inverted by $\iota^{\ast}$ is by definition an equivalence when evaluated on any affine, hence on any quasi-affine, and therefore on arbitrary spectral schemes. We deduce that the pair $\iota^{\ast}\dashv \iota_{\ast}$ supply mutually inverse equivalences.   
  Let us now demonstrate claim (1). Note first that there is an equivalence $\widehat{\twoCAlg} \simeq \Fun^{\mm{lex}}(\twoCAlg^{\omega, \op}, \widehat{\An})$. This induces a composite equivalence  \begin{equation} \begin{aligned}
    \Shv_{\mm{Zar}}(\CAlg^{\op}; \widehat{\twoCAlg})  \simeq & \Fun^{\mm{R}}(\widehat{\Shv}_{\mm{Zar}}(\CAlg^{\op}), \Fun^{\mm{lex}}(\twoCAlg^{\omega,\op}, \widehat{\An})) \\
    \simeq & \Fun^{\mm{R}, \lex}(  \widehat{\Shv}_{\mm{Zar}}(\CAlg^{\op}) \times \twoCAlg^{\omega, \op}, \widehat{\An}) \\
    \simeq & \Fun^{\lex}(\twoCAlg^{\omega, \op}, \Fun^{\mm{R}}(\widehat{\Shv}_{\mm{Zar}}(\CAlg^{\op}), \widehat{\An})) \\
    \simeq & \Fun^{\lex}(\twoCAlg^{\omega,\op}, \widehat{\Shv}_{\mm{Zar}}(\CAlg^{\op}))
    \end{aligned}
  \end{equation} which also holds if $\widehat{\Shv}_{\mm{Zar}}(\CAlg^{\op})$ (resp. for target $\widehat{\twoCAlg}$) is replaced with $\Fun(\CAlg, \widehat{\An})$ (resp. for target $\widehat{\twoCAlg}$). Moreover, these equivalences sit in a commutative square of adjoints:
  \[\xymatrix{
      \Fun(\CAlg, \widehat{\twoCAlg}) \ar@<0.7ex>[r] \ar[d]^{\simeq} &  \ar@<0.7ex>[l]^{\supseteq}\Shv_{\mm{Zar}}(\CAlg; \widehat{\twoCAlg}) \ar[d]^{\simeq} \\
      \Fun^{\lex}(\CAlg^{\omega, \op}, \Fun(\CAlg, \widehat{\An})) \ar@<0.7ex>[r] & \ar@<0.7ex>[l]^{\supseteq} \Fun^{\lex}(\CAlg^{\omega,\op}, \widehat{\Shv}_{\mm{Zar}}(\CAlg^{\op}))
    }
  \] see for example the proof of \cite[Lemma 3.27]{aokiHigherGeometriesTensor}. In particular, there is a left adjoint \tdef{sheafification with respect to the Zariski topology} 
    \[
      \Fun(\CAlg, \widehat{\twoCAlg}) \to \Shv_{\mm{Zar}}(\CAlg^{\op}; \widehat{\twoCAlg})
      \]
      whose right adjoint is the natural forgetful inclusion. The desired sheafification with respect to the \emph{canonical} topology is then given by the composite \[\Fun(\SpSch^{\op}, \widehat{\twoCAlg}) \to \Fun(\CAlg, \widehat{\twoCAlg}) \to \Shv_{\mm{Zar}}(\CAlg^{\op}; \widehat{\twoCAlg}) \simarrow \Shv_{\can}(\SpSch; \widehat{\twoCAlg}).\]

      To demonstrate claim (1), recall that the map $\alpha^{h}$ is an equivalence on affine schemes by \autoref{ex:relativeaffinespectra} and thus it must be sent to an equivalence after sheafification with respect to the canonical topology, using the previous paragraph. \autoref{ex:relativespecstructuresheaf} shows that it is already the case that \[\Gamma(\Spec^{\GZarst}_{\GZarsp}(-), \cc{O}) \in \Shv_{\can}(\SpSch; \twoCAlg)\] and hence $\alpha^{h}$ must be identified with the sheafification of its source.

  To conclude, it remains to identify the image of $\Perf_{\Gamma(-,\cc{O})}$ under sheafification with respect to the canonical topology. By construction, this is given by the composite \[\SpSch \hookrightarrow \widehat{\Shv}_{\mm{Zar}}(\CAlg)^{\op} \xrightarrow{\mathscr{Perf}} \widehat{\twoCAlg}\] where $\mathscr{Perf}$ is the image of the functor $\Perf \in \Fun(\CAlg, \widehat{\twoCAlg})$ under the composite
  \[
    \Fun(\CAlg, \widehat{\twoCAlg}) \to \Shv_{\mm{Zar}}(\CAlg^{\op}; \widehat{\twoCAlg}) \simeq \Fun^{\mm{R}}(\widehat{\Shv}_{\mm{Zar}}(\CAlg^{\op})^{\op}, \widehat{\twoCAlg}).
  \] Recall that the functor $\Perf$ already satisfies Zariski descent on $\CAlg^{\op}$, which follows from the fact that $\Mod$ satisfies Zariski descent \cite[Corollary D.6.3.3]{lurieSpectralAlgebraicGeometry} and the fact that passage to subcategories of dualizable objects is a limit preserving functor. Using the argument of \cite[Proposition 6.2.3.1]{lurieSpectralAlgebraicGeometry}, we find that the functor $\Perf: \Fun(\CAlg, \widehat{\An})^{\op} \to \widehat{\twoCAlg}$ factors through the sheafification with respect to the Zariski topology, and in particular that $\mathscr{Perf} \simeq \Perf$ when regarded as objects of $\Fun(\CAlg, \widehat{\twoCAlg})$. It follows that $\mathscr{Perf}$ can be identified with the composite
  \[
    \widehat{\Shv}_{\mm{Zar}}(\CAlg^{\op})^{\op} \subseteq \Fun(\CAlg, \widehat{\An})^{\op} \xrightarrow{\Perf} \widehat{\twoCAlg}
    \]
    and thus the image of $\Perf_{\Gamma(-,\cc{O})}$ under the sheafification with respect to the canonical topology on $\SpSch$ may be definitionally identified with the functor of perfect complexes.

    Finally, the only missing component for running the arguments above in the Dirac case is the fact that the functor $\Perf:\CAlg \to \twoCAlg$ in fact satisfies Dirac descent. For a given $R \in \CAlg$, the comparison map
    \[
       \Spec^{\GZarst}_{\GDirsp}\Spec^{\GDirsp}R \to \Spec^{\GDirsp}R 
    \]
    identifies the pushforward of the structure sheaf on the source with the unique assignment on $|\! \Spec^{\GDirsp}R|$ satisfying $\{D(f)\subseteq |\! \Spec^{\GDirsp}R|\} \mapsto \Perf_{R[f^{-1}]}$ where $D(f)$ is a basic quasicompact open subset of $|\! \Spec^{\GDirsp}R|$, see \cite[Section 3.E]{aokiHigherGeometriesTensor} for the structure theory and \cite[Theorem 4.48]{aokiHigherGeometriesTensor} for the identification of sheaves. Since this is in fact a sheaf, we may conclude that $\Perf$ satisfies descent for Dirac covers of $R$, and hence that $\Perf$ satisfies Dirac descent on $\CAlg^{\op}$. 
\end{proof}

\subsection{A reconstruction theorem of Balmer}
We now obtain \autoref{thmalpha:comparison} as a corollary of \autoref{thm:affineness}, and the classification of thick subcategories of perfect complexes over a Noetherian ring due to \cite{neemanChromaticTower1992}.

\begin{corollary}\label{thm:ntbrecon}
  There is a natural comparison transformation
  \[
  \gamma: (\Spec \Perf_{(-)}, \cc{R}_{\cc{O}_{\Perf_{(-)}}}) \Rightarrow \mm{id} 
\] of functors from $\SpSch_{\mm{qcqs}}$ to $\RTop^{\loc}_{\CAlg}$. Moreover:
\begin{enumerate}
\item For any $X \in \SpSch_{\mm{qcqs}}$, the comparison map $\gamma$ evaluated on $X$ identifies the structure sheaf on $X$ with the direct image $\gamma_{\ast}\cc{R}_{\cc{O}_{\Perf_{(-)}}}$.
\item For an ordinary qcqs scheme $X$ regarded as a $0$-truncated spectral scheme, the comparison map $\gamma$ evaluated on $X$ is an equivalence. 
\end{enumerate}
\end{corollary}
\begin{proof}
  From the natural identification of \autoref{thm:affineness}, it suffices to show the result with $\Spec \Perf_{(-)}$ replaced by $\Spec^{\cc{G}'}_{\cc{G}}$, where $\cc{G}' \coloneq \GZarst$ and $\cc{G} \coloneq \GZarsp$. Since the relative spectrum is left adjoint to restriction, one obtains a counit transformation
  \[
    (\Spec^{\cc{G}'}_{\cc{G}}, \cc{R}_{\cc{O}}) \to \mm{id} \in \Fun(\SpSch,\RTop^{\loc}_{\CAlg})^{[1]}
  \]
yielding the desired comparison map. For part a), note that for any cover by affine open subsets $\bigcup_{I} \Spec R_{i} \subseteq X$, one has an identification
  \[
    (\Spec^{\cc{G}'}_{\cc{G}}X, \cc{R}_{\cc{O}}) \times_{X} \coprod_{I} \Spec R_{i} \simeq \coprod_{I} (\Spec^{\cc{G}'}_{\cc{G}}\Spec R_{i}, \cc{R}_{\cc{O}_{i}}) \simeq \coprod_{I} (\Spec \Perf_{R_{i}}, \cc{R}_{\cc{O}_{i}})
    \]
    from \hyperref[prop:omnibusetale3]{\autoref*{prop:omnibusetale}.(3)}. We thus reduce to the case where $X = \Spec R$, where the pushforward of $\cc{R}_{\cc{O}_{{\Perf}_{R}}}$ is identified with the structure sheaf of $\Spec R$ by the proof of \cite[Theorem 4.48]{aokiHigherGeometriesTensor}.

    For part b), the claim similarly reduces to the affine case, and in fact further reduces to the case of $X = \Spec R$ for $R$ a classical Noetherian ring by using that $\Spec^{\cc{G}'}_{\cc{G}}\colon \CAlg \to \LTop(\cc{G}')$ is a left adjoint and that the restriction $\LTop(\cc{G}') \to \LTop(\cc{G})$ preserves filtered colimits \cite[Corollary 1.5.4]{lurieDerivedAlgebraicGeometryV}. This case is demonstrated in Neeman \cite{neemanChromaticTower1992}.
\end{proof}

\begin{warn}
    Note that the proof above recovers the ordinary scheme $X$ regarded as an object of $\SpSch$ and not simply as an ordinary ringed space. In particular, the sections of the structure sheaf on any given open set may not be concentrated in $\pi_{0}$\footnote{The example $X = \bb{A}^{2} \setminus 0$ supplies a case where $\pi_{-1}\Gamma(X, \cc{O}) =: H^{1}(X, \cc{O})$ is nontrivial.}. The classically ringed space associated to $X$ may be recovered by taking the sheafification of the assignment $U \mapsto \pi_{0}\cc{O}_{X} \in \CAlg^{\heart}$, which is the context originally considered in  \cite{balmerPresheavesTriangulatedCategories2002}.
\end{warn}

\begin{remark}
  The previous strategy of proof demonstrates that for $X$ a qcqs spectral scheme, the underlying set of the Zariski spectrum stratifies into:
  \begin{enumerate}
  \item A ``geometric'' direction, corresponding to the support theory of the underlying scheme itself.
  \item A ``homotopical'' direction, corresponding to the failure of the comparison maps $|\! \Spec \Perf_{\cc{O}_{X,x}}\! | \to |\! \Spec \cc{O}_{X,x}|$ to be injective or surjective at points $x \in X$.
  \end{enumerate}
  For example, the result of \autoref{thm:ntbrecon} holds for regular, noetherian, locally even periodic schemes (for example, an oriented elliptic curve over a regular noetherian even periodic $\bb{E}_{\infty}$-ring) by reduction to the case of regular noetherian even periodic $\bb{E}_{\infty}$-ring spectra, which is handled in \cite[\S 2]{mathewThickSubcategoryTheorem2015a}. 
\end{remark}

\begin{remark}
  It is often the case that there are interesting specializations between points in distinct fibers of the comparison map $|\! \Spec \Perf_{X}\! | \to |X|$, see for example, the ``blueshifting'' behaviour exhibited in the fibers of the map $|\! \Spec \Perf_{\bb{S}_{G}^{G}}\! | \to |\! \Spec \pi_{0}\bb{S}^{G}_{G}|$ where $G$ is a finite group and $\bb{S}_{G}^{G}$ refers to the categorical fixed points of the genuine $G$-sphere \cite[\S 13]{sanders2025tensortriangulargeometryfully}.
\end{remark}

Finally, let us mention the differences between our approach to the computation of $|\! \Spec \Perf_{X}|$ for an ordinary qcqs scheme $X$ and the approaches adopted by Thomason and Balmer in the noetherian case.

\begin{remark}\label{rem:differences1}
We briefly recall the components of the argument in \cite{thomasonClassificationTriangulatedSubcategories1997}, which is a classification of the thick tensor ideals of $\Perf_{X}$. By the identification between the underlying space of $\Spec \Perf_{X}$ and the Balmer spectrum $\Spc \Perf_{X}$, this is equivalent to the desired computation, see for example \cite[\S 4.2]{aokiHigherGeometriesTensor} for a quick overview of the same.

First, let $X$ be a noetherian qcqs potentially non-affine scheme. The classification of thick tensor ideals consists of two steps:
\begin{enumerate}
    \item Demonstrating that there are objects of $\Perf_{X}$ whose cohomology groups are supported on any choice of closed subscheme $Z\subseteq X$. 
    \item Showing that any thick tensor ideal $\cc{I} \subseteq \Perf_{X}$ is determined by the union of the supports of cohomology groups of complexes in $\cc{I}$.
\end{enumerate}

The second component follows via a nilpotence theorem for $\Perf_{X}$, extending the results of \cite{hopkinsGlobalMethodsHomotopy1987} and \cite{neemanChromaticTower1992}. As both his and our approaches rely on these results in the affine case, the distinctions here are at most cosmetic. 

In order to demonstrate the first component, the author invokes a $K$-theory extension theorem to iteratively build complexes with the desired support by inducting from affines, see \cite[Lemma 4.1.8]{KockPitsch17} for a brief overview of his method along with a systematization using the ``Reduction Principle'' of \cite{bondalGeneratorsRepresentabilityFunctors2003}. By constrast, we do not construct complexes supported on every possible closed subset of $X$. In fact, the comparison transformation $\gamma_{X}: \Spec \Perf_{X} \to X$ often fails to be surjective for a general spectral scheme, owing to the fact that complexes with prescribed support may not always be constructed. For an ordinary scheme, the existence of these complexes is deduced \emph{a posteriori} from the construction of a functorial comparison map and a direct reduction to the affine case.\footnote{In some sense, this too is almost cosmetic: the construction of requisite complexes is hidden in the iterative process of \autoref{thm:affineness}, but we will not expand upon this here.}
\end{remark}

\begin{remark}\label{rem:differences2}
  We now remark on the reconstruction of the locally ringed space $X$ by Balmer in \cite{balmerPresheavesTriangulatedCategories2002}. Here, he writes down a map
  \[
    |X| \to |\! \Spec \Perf_{X}\! |
  \] on the level of topological spaces by the characterization of $|\! \Spec \Perf_{X}|$ as a \emph{universal support theory}. His map in fact goes in the \emph{opposite} direction to the map $\gamma_{X}$, and in fact supplies the inverse equivalence to ours, which he deduces from Thomason's classification above. Note that since there exist spectral schemes $X$ for which the comparison transformation $\gamma_{X}$ is not surjective, it is usually not possible to construct Balmer's map from $|X| \to |\! \Spec \Perf_{X}\! |$, since such a map can be shown to split $\gamma_{X}$. 
\end{remark}

\subsection{The Thomason property for spectral schemes}

The comparison transformation $\gamma$ of \autoref{thm:ntbrecon} encodes several interesting features about the spectral scheme $X$ under consideration. The aim of this brief subsection is to document one such property, assuming the language and basic classification results of tt-geometry as discussed in \cite{balmerSpectrumPrimeIdeals2005}.

\begin{definition}
  Let $X$ be an arbitrary spectral scheme. We say $X$ \tdef{satisfies the Thomason property} if for any quasicompact open subset $U \subseteq X$ with closed complement $Z$, there exists a perfect complex $P \in \Perf_{X}$ such that $Z = \{x \in X \mid P_{x} \not\simeq 0 \in \Perf_{\cc{O}_{X,x}}$. We say in this case that $P$ is \tdef{supported on $Z$}.
\end{definition}

\begin{lemma}
  For any quasicompact spectral scheme $X$, the comparison map
  \[
    \gamma : (\Spec \Perf_{X}, \cc{R}_{\cc{O}_{\Perf_{X}}}) \to X
    \] yields a quasicompact map of spaces $|\! \Spec \Perf_{X} \! | \to |X|$. 
\end{lemma}
\begin{proof}
  Since both spaces are quasicompact, quasiseparated, and admit a basis of quasicompact opens, the question is affine-local on $X$. The base-change of $\gamma_{X}$ along any affine open $\Spec R \subseteq X$ gives the comparison map $|\! \Spec \Perf_{R}\! | \to |\! \Spec R |$ discussed in \cite[\S 5]{balmerSpectraSpectraSpectra2010} by the results of \cite[\S 4.D]{aokiHigherGeometriesTensor}. As each of these maps is quasicompact by \cite[Remark 5.7]{balmerSpectraSpectraSpectra2010}, we deduce the claim.
\end{proof}

\begin{proposition}
  For any $X$ as above, the comparison map $\gamma$ is surjective if and only if $X$ satisfies the Thomason property.
\end{proposition}
\begin{proof}
  First, observe that for any point $z \in X$ we may identify
  \[\gamma^{-1}(\{z\}) = \bigcap\nolimits_{z \in U_{i}} \gamma^{-1}(U_{i}) = \varprojlim_{z \in U_{i}}|\! \Spec \Perf_{U_{i}} \!| = |\! \Spec \Perf_{\cc{O}_{X,z}}\! |\] where the $U_{i}$ are any system of quasicompact opens which all intersect in $z$. The middle identification is an application of \hyperref[prop:omnibusetale3]{\autoref*{prop:omnibusetale}.(3)}, while the right-most identification comes from the fact that the functor $|\! \Spec(-)|$ preserves the filtered colimit $\varinjlim{z \in U_{i}} \Perf_{U_{i}} \simeq \Perf_{\cc{O}_{X,z}}$, once again by employing \cite[Corollary 1.5.4]{lurieDerivedAlgebraicGeometryV}.

  We now proceed to the claim. Suppose first that $\gamma$ is surjective. Then for any quasicompact open $U \subset X$ with closed complement $Z$, the preimage $\gamma^{-1}(U) \subseteq |\! \Spec \Perf_{X}\! |$ is also quasicompact open with closed complement $\gamma^{-1}(Z)$. By the classification property of the Balmer spectrum \cite{balmerSpectrumPrimeIdeals2005}, there is a perfect complex $P \in \Perf_{X}$ whose support $\mathrm{supp}(Z)$ in $|\! \Spec \Perf_{X}\! |$ is exactly $\gamma^{-1}(Z)$. We claim that $P$ is supported exactly on $Z$: to see this, note that the support of $P$ in any $|\! \Spec \Perf_{\cc{O}_{X,z}}\! |$ is identified exactly with the intersection of the support of $P$ with $\gamma^{-1}(\{z\})$ by the above. This is nonempty if and only if $z \in Z$, implying $P$ is supported at $z$ if and only if $z \in Z$.
  
  Now suppose that $\gamma$ is not surjective, and let $z \in X$ be a point that avoids the image of $\gamma$. Note that if there exists a perfect complex $P \in \Perf_{X}$ which is supported on a closed set containing $z$, the image of $P$ in $\Perf_{\cc{O}_{X,x}}$ must be a nonzero perfect complex satisfying the same condition. We are reduced to demonstrating that no such complex can exist in $\Perf_{\cc{O}_{X,z}}$. First, note that by the first paragraph, the comparison map
  \[\gamma : |\! \Spec \Perf_{\cc{O}_{X,z}} \! | \to |\! \Spec \cc{O}_{X,z} | \] does not contain $z$. Since the source is quasicompact and $z$ is a closed point of the target, there must exist some quasicompact open subset $U \subseteq |\! \Spec \cc{O}_{X,z} |$ which contains the entire image of $\gamma$. Applying \hyperref[prop:omnibusetale3]{\autoref*{prop:omnibusetale}.(3)}, we learn that the inclusion
  \[
    \Perf_{U} \simeq U \times_{\Spec \cc{O}_{X,z}} \Spec \Perf_{\cc{O}_{X,z}} \subseteq \Spec \Perf_{\cc{O}_{X,z}}
  \]
  must have in fact been an equivalence, and in particular that the restriction map $\Perf_{\cc{O}_{X,z}} \to \Perf_{U}$ must have been an equivalence. It follows that there is no complex $P$ which is supported on the complement of $U$, and hence the Thomason condition fails.
\end{proof}

\begin{remark}
  Curiously, the statement above demonstrates that the validity of Balmer's approach to the classification problem (discussed in \autoref{rem:differences2}) actually relied on the Thomason property of $X$. Namely, we have that for a qcqs spectral scheme $X$, if the \emph{homological support function} of \cite[Notation 2.7]{balmerPresheavesTriangulatedCategories2002} given by
  \[
    P \in \Perf_{X} \mapsto \{x \in X \mid P_{x} \in \Perf_{\cc{O}_{X,x}} \not\simeq 0\} \subseteq X
    \] actually defines a \emph{support datum} in the usual tt-geometric sense, then the comparison map $\gamma$ evaluated on $X$ is surjective and hence $X$ must satisfy the Thomason condition.
\end{remark}
\section{Geometrization of 2-rings}\label{sec:geometrization}

This section is dedicated to the proof of \autoref{thmalph:geometrization}, recorded as \autoref{thm:geom} below.
\subsection{An aside on unigenicity} Before proceeding, it will be helpful later to systematize the exact relation between rings and $2$-rings. Recall the following notion, discussed in \cite{hoveyAxiomaticStableHomotopy1997a} and \cite[\S 8]{sanders2025tensortriangulargeometryfully}.

\begin{definition}
  $\cc{K} \in \twoCAlg$ is \tdef{unigenic} if the smallest thick subcategory containing $\unit_{\cc{K}}$ is $\cc{K}$ itself. We write $\mdef{\twoCAlg^{\mm{uni}}} \subset \twoCAlg$ to denote the full subcategory of unigenic 2-rings.
\end{definition}

\begin{recollection}\label{cons:monogenicreflection}
  Recall the adjunction of \autoref{def:endfunctor}
  \[
    \Perf: \CAlg \rightleftarrows \twoCAlg: \cc{R_{(-)}}
  \]
  whose  left adjoint is fully faithful.  By \cite[7.1.2.7]{lurieHigherAlgebra} and \cite[7.6]{mathewNilpotenceDescentEquivariant2017}, the categories of the form $\Mod_{R} \in \CAlg(\mm{Pr^{L,\omega}_{st}})$ for $R \in \CAlg$ are uniquely categorized by the fact that their subcategories of compact objects are unigenic, from which it follows that that the embedding $\Perf:\CAlg \to \CAlg(\Cat^{\mm{perf}})$ has essential image exactly the full subcategory of unigenic 2-rings. Composing the adjunction above with the equivalence $\CAlg \simeq \twoCAlg^{\mm{uni}}$, we obtain an adjunction of the form
\[ F: \twoCAlg^{\mm{uni}} \rightleftarrows \twoCAlg :G \]  where $F$ is the inclusion. We will deduce a formula for $G$ below. 
 \end{recollection}

 \begin{lemma}\label{lem:monogenicreflection}
   The right adjoint $G'$ to the inclusion $\twoCAlg^{\mm{uni}} \hookrightarrow \twoCAlg$ is given by the functor $\cc{K} \mapsto \mm{Thick}(\unit_{\cc{K}})$, and the counit is identified with the canonical inclusion $\mm{Thick}(\unit_{\cc{K}}) \subseteq \cc{K}$. 
 \end{lemma}
 \begin{proof}
   Let $\cc{K}\in \twoCAlg$ arbitrary. As the left adjoint is an inclusion, \cite[5.2.2.7]{lurieHigherToposTheory2009} implies that the counit  $\epsilon: G(\cc{K}) \to \cc{K}$ is a terminal object of $\twoCAlg^{\mm{uni}}_{-/\cc{K}}$. For $\unit_{\cc{K}} \in \cc{K}$, consider the thick subcategory $\mm{Thick}(\unit_{\cc{K}}) \hookrightarrow \cc{K}$, and note that this inclusion is a morphism in $\twoCAlg$ with unigenic domain. Furthermore, the image of any $F:\cc{L}\to \cc{K} \in \twoCAlg^{\mm{uni}}_{-/\cc{K}}$ must be contained in $\mm{Thick}(\unit_{\cc{K}})$, as $F(\mm{Thick}(\mathscr{S})) \subseteq \mm{Thick}(F(\mathscr{S}))$ for $\mathscr{S}\subseteq \cc{L}$ any set of objects. Thus, the counit map admits a lift through $G(\cc{K}) \to \mm{Thick}(\unit_{\cc{K}})$. Any such lift is split by lifting the canonical inclusion $\mm{Thick}(\unit_{\cc{K}}) \to \cc{K}$ through $\mm{Thick}(\unit_{\cc{K}}) \to G\cc{K} \to \mm{Thick}(\unit_{\cc{K}})$, which composes to the identity. As any retract of a terminal object must itself be terminal, we have an equivalence $\mm{Thick}(\unit_{\cc{K}}) \simeq G(\cc{K})$ and the counit is identified with the canonical inclusion. 
 \end{proof}

 \begin{definition}
   We refer to the right adjoint $G'$ above as the functor of \tdef{unitation}.
 \end{definition}

 \subsection{The main result}
 
The following is \autoref{coralph:univprop}, which is obtained as an immediate consequence of \autoref{thm:affinenessofschemes}.

\begin{corollary}\label{cor:univproperty}
  For any $\cc{K} \in \twoCAlg_{\rig}$ and qcqs spectral scheme $X$, there is an identification
  \[
    \map_{\twoCAlg_{\rig}}(\Perf_{X}, \cc{K}) \simeq \map_{\RTop^{\loc}_{\CAlg}}((\Spec \cc{K}, \cc{R}_{\cc{O_{K}}}), X)
  \]
  The same holds if $X$ is a Dirac spectral scheme and $\RTop^{\loc}_{\CAlg}$ is replaced by $\RTop^{\Dir}_{\CAlg}$.
\end{corollary}

\begin{proof}
  By \autoref{thm:zariski2subcanonical} the functor $\Spec:\twoCAlg^{\op}_{\rig} \to \RTop^{\loc}_{\twoCAlg}$ is fully faithful. The claimed result thus reduces to the identification
  \[
    \map_{\RTop^{\loc}_{\twoCAlg}}(\Spec \cc{K}, \Spec \Perf_{X}) \simeq \map_{\RTop^{\loc}_{\CAlg}}((\Spec \cc{K}, \cc{R}_{\cc{O_{K}}}), X)
  \]
  which is the content of \autoref{thm:affinenessofschemes}. The same argument holds in the Dirac case.
\end{proof}

We may proceed to the proof of \autoref{thmalph:geometrization}.

\begin{theorem}\label{thm:geom}
  Let $\cc{K} \in \twoCAlg_{\rig}$ be a rigid $2$-ring and $X \in \SpSch$ admitting an abstract equivalence
  \[
    \alpha: (\Spec \cc{K}, \cc{R_{O_{K}}}) \simeq X \in (\RTop^{\loc}_{\CAlg})^{[1]}
  \]
  Then there is a fully faithful functor $\Perf_{X} \to \cc{K}$ realizing the equivalence $\alpha$ upon passage to Zariski spectra. The same result holds in the Dirac case.
\end{theorem}
\begin{proof}
  The assumption forces $X$ to be qcqs, from which \autoref{cor:univproperty} implies that the map $\alpha$ induces a functor $\Perf_{X} \to \cc{K}$. Now let $U \subseteq X$ be any quasicompact quasi-affine open subset. By the identification of \autoref{thm:zariskiisbalmer}, there is an associated commuting square of Karoubi quotients 
  \[\xymatrix{
      \Perf_{X} \ar[d] \ar[r] & \cc{K} \ar[d]\\
      \Perf_{U} \ar[r] &  \cc{O_{K}}(U).
    }\]
  Moreover, since $\alpha$ is assumed to be an equivalence of spectrally ringed spaces, it induces an identification \[\Gamma(U, \cc{O_{X}}) \simeq \Gamma(U, \cc{R_{O_{K}}}) \simeq \cc{R_{\cc{O_{K}}}(U)}\] and hence the associated functor $\Perf_{\Gamma(U, \cc{O_{X}})} \to \cc{O_{K}}(U)$ factors through an equivalence
  \[
    \Perf_{\Gamma(U, \cc{O_{X}})} \simeq \Perf_{\cc{R}_{O_{\cc{K}}(U)}}
    \] with the unitation of $\cc{O}_{\cc{K}}(U)$. Invoking \autoref{lem:monogenicreflection}, we learn that the map $\Perf_{\Gamma(U, \cc{O_{X}})} \to \cc{O_{K}}(U)$ must be fully faithful and in fact induces an equivalence between the source and the thick subcategory generated by the unit in the target. Furthermore, since $U$ is assumed to be qc quasi-affine, the canonical map $\Perf_{\Gamma(U, \cc{O_{X}})} \to \Perf_{U}$ must be an equivalence.  Altogether, we obtain that the composite map \[\Perf_{\Gamma(U, \cc{O_{X}})} \simeq \Perf_{U} \to \cc{O_{K}}(U)\] is fully faithful. Since the qc quasi-affine open subsets form a basis of $X$, it follows that the map $\Perf_{X} \to \cc{K}$ may be presented as a limit of fully faithful functors
  \[
    \varprojlim_{U \subseteq X \text{ qc quasi-affine}} [\Perf_{U} \to \cc{O_{K}}(U)] \in \twoCAlg_{\rig}^{[1]}
  \]
  and is thus itself fully faithful. The argument does not change for the Dirac case.

  The statement that the passage to Zariski spectra realizes the map $\alpha$ arises from the fact that the composite 
  \[
  (\Spec \cc{K}, \cc{R_{O_K}}) \to (\Spec \Perf_{X}, \cc{R}_{\cc{O}_{\Spec \Perf_{X}}}) \to X 
  \]
  in $\RTop^{\loc}_{\CAlg}$ realizes the equivalence $\alpha$, by \autoref{thm:affinenessofschemes} and the universal property of the relative spectrum. Moreover, the map of spaces $|\! \Spec \cc{K}| \to |\! \Spec \Perf_{X}|$ is a strong spectral quotient map by \cite[Theorem 4.1]{sanders2025tensortriangulargeometryfully}, implying in particular that it is a homeomorphism if and only if it is injective \cite[Corollary 2.26]{sanders2025tensortriangulargeometryfully}. However, since $\alpha$ is an equivalence, it is necessarily injective, yielding the claim.
\end{proof}

We collect some key examples below.

\begin{example}[Stable module categories]\label{ex:stablemodule}
  Let $G$ be a finite group of order divisible by $p$, and $k$ be a field of characteristic $p$. The \tdef{stable module category} of $G$ over $k$ is the object $\mdef{\mm{St}_{kG}} \in \twoCAlg_{\rig}$ defined by the Verdier quotient
  \[
    \mm{St}_{kG} \simeq \mm{Rep}_{G}(k)/\mm{Proj}_{G}(k)
  \]
  where $\mm{Proj}_{G}(k)$ is the thick subcategory generated by the image of the projective representations in $\mm{Rep}_{G}(k)$ (here by $\mm{Rep}$, we really mean the derived category of finite-dimensional representations of $G$ over $k$). The Balmer spectrum of $\mm{St}_{kG}$ is computed in \cite{Benson1997}, where it is shown to be homeomorphic to $\mm{Proj}~H^{\ast}(G, k)$.  \cite[Proposition 8.8]{balmerSpectraSpectraSpectra2010} demonstrates more, namely this homeomorphism is realized on underlying topoi via the comparison map
  \[
    (\Spec \mm{St}_{kG}, \cc{R_{O}}) \to \Spec^{\GDirsp}\cc{R}_{\mm{St}_{kG}} \in \RTop^{\Dir}_{\CAlg}
  \]
  and the same argument easily shows that the above is an equivalence in locally spectrally ringed spaces. However, as a spectrally ringed space, the object $\Spec^{\GDirsp}\cc{R}_{\mm{St}_{kG}}$ is in fact an \emph{ordinary} spectral scheme, given by the space $\mm{Proj}~H^{\ast}(G, k)$ equipped with the sheaf of spectra
  \[
    \{U(f) \subseteq \mm{Proj}~H^{\ast}(G, k)\} \mapsto \cc{R}_{\mm{St}_{kG}}[f^{-1}].
  \]
  We refer to this spectral scheme as the \tdef{spectral support variety} of $G$ over $k$, and use the notation $\mdef{\cc{V}_{G}}$ to denote it. By the result of \autoref{thm:geom}, there is a fully faithful functor $\Perf_{\cc{V}_{G}} \to \mm{St}_{kG}$. Moreover, it can be shown that this embedding is functorial in the group $G$.
\end{example}

The primary motivation for treating the Dirac spectral setting above is the following example, which shows that Dirac spectral schemes ought also abound in nature.

\begin{example}(Permutation modules)\label{ex:permmods}
  Let $G$ be a group and $R \in \CAlg^{\heart}$ be an ordinary commutative Noetherian ring. A \tdef{permutation $R[G]$-module} is an $R[G]$ module of the form $R[X]$ for $X \in \mm{Set}^{BG}$ a $G$-set. Let $\mdef{\mm{perm}(G, R)} \subseteq \Mod^{\heart}_{R[G]}$ denote the additive category of finitely generated permutation $R[G]$-modules. The tensor-product of $R[G]$-modules restricts to $\mm{perm}(G,k)$, and the latter is thus an additively monoidal category.

  Consider the bounded homotopy category of chain complexes in $\mm{perm}(G,k)$, denoted $\mdef{\cc{K}(G,R)}$. This is a rigid tensor-triangulated category whose Balmer spectra are computed in \cite{balmerGeometryPermutationModules2025} for the case where $R$ is a field, and in  \cite{dubeyBalmerSpectrumIntegral2025} for the case where $R$ is characteristic $p$ or a $p$-torsion free Noetherian ring.

  For $G = E$ an elementary abelian group, the Balmer spectrum $\Spc \cc{K}(G,R)$ equipped with the unique structure sheaf satisfying
  \[
    \cc{O}_{E}: \{U(a) \subseteq \Spc \cc{K}(G,R) \text{ quasicompact open}\} \mapsto \pi_{\ast}\End_{\cc{K}(G,R)/\langle a \rangle}(\unit)
  \]
  is shown to be a Dirac scheme in \cite[Corollary 15.4]{balmerGeometryPermutationModules2025}, \cite[Corollary 10.12]{dubeyBalmerSpectrumIntegral2025}. For any given enhancement of $\cc{K}(G,R)$ to the structure of a $2$-ring, this result will also demonstrate that the Dirac-locally spectrally ringed $\infty$-topos
  \[
    \mdef{\cc{V}_{G}^{\mm{perm}}} \coloneq (\Spec \cc{K}(G,R), \cc{R_{O}}) \in \RTop^{\Dir}_{\CAlg}
  \]
  is a Dirac spectral scheme, through which \autoref{thm:geom} will supply a fully faithful embedding $\Perf_{\cc{V}^{\mm{perm}}_{G}} \hookrightarrow \cc{K}(G,R)$. We hope to return to this observation in future work. 
\end{example}

Finally, we remark that the same proof as in \autoref{thm:geom} actually supplied a slightly stronger statement.

\begin{theorem}\label{thm:slightlybettergeom}
    Let $\cc{K} \in \twoCAlg_{\rig}$ be a rigid $2$-ring and $X \in \SpSch_{\mm{qcqs}}$ admit a map 
    \[
    \alpha: (\Spec \cc{K}, \cc{R}_{\cc{O}_K}) \to X \in (\RTop^{\loc}_{\CAlg})^{[1]}
    \]
    such that the associated map $\cc{O}_{X} \to \alpha_{\ast} \cc{R}_{O_K}$ is an equivalence in $\Shv(X;\CAlg)$. Then there is a fully faithful functor $\Perf_{X} \to \cc{K}$.
\end{theorem}

\appendix

\section{Some non-affine examples}\label{sec:non-affines} Below, we collect several examples of non-affine $2$-schemes. First, note that the relative spectrum furnishes an easy way to promote examples from schemes into non-affine moduli.

\begin{example}
  Given any spectral scheme $X$, the relative spectrum $\Spec^{\cc{G}'}_{\cc{G}}X$ is always a $2$-scheme by the argument of \autoref{thm:affinenessofschemes}. However, if $X$ is not qcqs, then this construction cannot be guaranteed to be affine. For example, if $X$ is a non-qcqs ordinary scheme with relative spectrum $(\Spec^{\cc{G}'}_{\cc{G}}X, \cc{O})$, the map
  \[
    (\Spec^{\cc{G}'}_{\cc{G}}X, \cc{R}_{\cc{O}}) \to X
  \]
  is always an equivalence of locally spectrally ringed spaces by the same argument as in \autoref{thm:ntbrecon}. By \autoref{thm:affineness}, $\Spec^{\cc{G}'}_{\cc{G}}X$ could not have itself been an affine $2$-scheme. This above example is rather interesting from the point of view of the tt-geometry of $\Perf_{X}$: note that \autoref{lem:sheafifyingperf} supplies an identification of the form
\[\Gamma(\Spec^{\cc{G}'}_{\cc{G}}X, \cc{O}) \simeq \Perf_{X} \in \twoCAlg_{\rig}.\] In particular, the affinization of the relative spectrum gives a map of locally $2$-ringed spaces
\[
  (\Spec^{\cc{G}'}_{\cc{G}}X, \cc{R}_{\cc{O}}) \to \Spec \Gamma(\Spec^{\cc{G}'}_{\cc{G}}X, \cc{O}) \simeq \Spec \Perf_{X}
\] which, on underlying spaces, gives a map $| X | \to |\! \Spec \Perf_{X}\! |$. By assumption on $X$, this map is not a homeomorphism (since the target is qcqs) and it need not even be surjective. 
\end{example}

Another class of $2$-schemes arise from considering \emph{quasi-affine} $2$-schemes. These turn out to also admit interesting and subtle examples.

\begin{definition}
  We say a $2$-scheme $(\cc{X}, \cc{O}_{\cc{X}})$ is \tdef{quasi-affine} if it is equivalent to an open subset of an affine $2$-scheme.
\end{definition}

Examples of the above include potentially non quasi-compact open subsets of affine schemes.

\begin{example}
  Recall that the spectrum of finite $p$-local spectra $\Sp_{(p)}^{\omega}$ for a prime $p$ is given by \[ |\! \Spec \Sp_{(p)}^{\omega}| \cong \mathbb{N}_{\geq 0} \cup \{\infty\}\] equipped with the order topology. $\Spec \Sp_{(p)}^{\omega}$ thus contains a quasi-affine $2$-scheme $\{\infty\}^{c}$, the complement of the closed point. The global sections of the structure sheaf on this subset may be computed as
  \[
    \Gamma(\{\infty\}^{c}, \cc{O}_{\Sp^{\omega}_{(p)}}) \simeq \varprojlim\nolimits_{\mathbb{N}} L_{n}^{f}\Sp^{\omega}
  \]
  where $L_{n}^{f}$ correspond to the finite localizations away from type $n$ complexes. This limit is not known, although it is conjectured to admit $\Sp_{(p)}^{\omega}$ as a retract via the unique map \cite{Blumberg2026}. It is known that the natural restriction functor
  \[
    \Sp_{(p)}^{\omega} \to \Gamma(\{\infty\}^{c}, \cc{O}_{\Sp^{\omega}_{(p)}})
  \]
  must be both conservative and faithful on morphisms, since the chromatic convergence theorem gives the composite identification
  \[
    X \to \varprojlim\nolimits_{\mathbb{N}}L_{n}^{f}X \to \varprojlim\nolimits_{\mathbb{N}}L_{n}X \simeq X.
  \] From this, we deduce that in the composite map of $2$-schemes 
  \[
    \Gamma(\{\infty\}^{c}, \cc{O}_{\Sp^{\omega}_{(p)}}) \to \Spec \Gamma(\{\infty\}^{c}, \cc{O}_{\Sp^{\omega}_{(p)}}) \to \Spec \Sp^{\omega}_{(p)}
    \] the second map must be surjective on underlying spaces, and hence the first map must fail to be surjective. A more thorough understanding of the behaviour of this $2$-scheme may shed light on the phenomena at play here.
\end{example}

\begin{remark}
  Applying algebraic $K$-theory pointwise to the structure sheaf of a $2$-scheme $(\cc{X}, \cc{O}_{\cc{X}})$ and sheafifying gives rise to a sheaf of spectra. For the previous example, the global sections of the resultant $K$-theory sheaf are related to \emph{Waldhausen's Chromatic Convergence Conjecture} \cite[Conjecture 1.9]{Blumberg2026}.
\end{remark}

\printbibliography

@book{lurieHigherAlgebra,
  title = {Higher {{Algebra}}},
  shorthand = {HA},
  author = {Lurie, Jacob},
  date = {2017},
  url = {https://www.math.ias.edu/~lurie/papers/HA.pdf}
}

@book{lurieHigherToposTheory2009,
  title = {Higher {{Topos Theory}} ({{AM-170}}) ({{Annals}} of {{Mathematics Studies}})},
  author = {Lurie, Jacob},
  shorthand = {HTT},
  date = {2009-07},
  publisher = {{Princeton University Press}},
  location = {{Ewing, New Jersey, U.s.a.}},
  isbn = {978-0-691-14049-0},
}

@book{lurieSpectralAlgebraicGeometry,
  title = {Spectral Algebraic Geometry (Version: {{February}} 2018)},
  shorthand = {SAG},
  shorttitle = {Spectral Algebraic Geometry (Version},
  author = {Lurie, J.},
  url = {https://www.math.ias.edu/~lurie/papers/SAG-rootfile.pdf},
}

@article{mathewNilpotenceDescentEquivariant2017,
  title = {Nilpotence and Descent in Equivariant Stable Homotopy Theory},
  author = {Mathew, Akhil and Naumann, Niko and Noel, Justin},
  date = {2017-01},
  journaltitle = {Advances in Mathematics},
  shortjournal = {Advances in Mathematics},
  volume = {305},
  eprint = {1507.06869},
  eprinttype = {arxiv},
  primaryclass = {math},
  pages = {994--1084},
  issn = {00018708},
  doi = {10.1016/j.aim.2016.09.027},
  url = {http://arxiv.org/abs/1507.06869},
  urldate = {2022-07-18},
  archiveprefix = {arXiv},
  langid = {english},
 }

@misc{aokiHigherGeometriesTensor,
  title={Higher Zariski Geometry},
  shorthand={HZG},
  author={Ko Aoki and Tobias Barthel and Anish Chedalavada and Tomer Schlank and Greg Stevenson},
  year=2025,
  eprint={2508.11621},
  archivePrefix={arXiv},
  primaryClass={math.AG},
  url={https://arxiv.org/abs/2508.11621}, 
}

@article{balmerPicardGroupsTriangular2010,
  title = {Picard groups in triangular geometry and applications to modular representation theory},
  volume = {362},
  ISSN = {0002-9947},
  url = {http://dx.doi.org/10.1090/S0002-9947-10-04949-4},
  DOI = {10.1090/s0002-9947-10-04949-4},
  number = {07},
  journal = {Transactions of the American Mathematical Society},
  publisher = {American Mathematical Society (AMS)},
  author = {Balmer,  Paul},
  year = {2010},
  month = 2,
  pages = {3677–3690}
}

@article{balmerSpectraSpectraSpectra2010,
  title = {Spectra, Spectra, Spectra – {{Tensor}} Triangular Spectra versus {{Zariski}} Spectra of Endomorphism Rings},
  author = {Balmer, Paul},
  date = {2010-07-02},
  journaltitle = {Algebraic \& Geometric Topology},
  shortjournal = {Algebr. Geom. Topol.},
  volume = {10},
  number = {3},
  pages = {1521--1563},
  issn = {1472-2739, 1472-2747},
  doi = {10.2140/agt.2010.10.1521},
  url = {http://www.msp.org/agt/2010/10-3/p10.xhtml},
  urldate = {2022-10-06},
  langid = {english},
}

@article{Mathew2016,
  title = {The Galois group of a stable homotopy theory},
  volume = {291},
  ISSN = {0001-8708},
  url = {http://dx.doi.org/10.1016/j.aim.2015.12.017},
  DOI = {10.1016/j.aim.2015.12.017},
  journal = {Advances in Mathematics},
  publisher = {Elsevier BV},
  author = {Mathew,  Akhil},
  year = {2016},
  month = 3,
  pages = {403–541}
}

@article{balmerPresheavesTriangulatedCategories2002,
  title = {Presheaves of Triangulated Categories and Reconstruction of Schemes},
  author = {Balmer, P.},
  year = {2002},
  month = {11},
  journal = {Mathematische Annalen},
  volume = {324},
  number = {3},
  pages = {557--580},
  issn = {0025-5831, 1432-1807},
  doi = {10.1007/s00208-002-0353-1},
  langid = {english},
}

@misc{gaitsgorySheavesCategoriesNotion2014,
  title = {Sheaves of Categories and the Notion of 1-Affineness},
  author = {Gaitsgory, Dennis},
  year = {2014},
  month = {8},
  eprint = {1306.4304},
  eprinttype = {arxiv},
  primaryclass = {math},
  publisher = {{arXiv}},
  doi = {10.48550/arXiv.1306.4304},
  archiveprefix = {arXiv},
}

@article{ben-zviIntegralTransformsDrinfeld2010a,
  title = {Integral Transforms and {{Drinfeld}} Centers in Derived Algebraic Geometry},
  author = {{Ben-Zvi}, David and Francis, John and Nadler, David},
  year = {2010},
  month = {4},
  journal = {Journal of the American Mathematical Society},
  volume = {23},
  number = {4},
  pages = {909--966},
  issn = {0894-0347, 1088-6834},
  doi = {10.1090/S0894-0347-10-00669-7},
   langid = {english},
}

@article{lurieDerivedAlgebraicGeometryV,
  shorthand = {DAGV},
  title={Derived Algebraic Geometry {V}: Structured spaces},
  author={Lurie, Jacob},
  journal={Draft available at the author's website.},
  year={2011},
}

@article{lurieDerivedAlgebraicGeometry2011a,
  title = {Derived {{Algebraic Geometry VII}}: {{Spectral Schemes}}},
  shorthand = {DAGVII},
  author = {Lurie, Jacob},
  year = {2011},
  month = {11},
}

@article{bondalGeneratorsRepresentabilityFunctors2003,
  title = {Generators and {{Representability}} of {{Functors}} in {{Commutative}} and {{Noncommutative Geometry}}},
  author = {Bondal, A. and {Van den Bergh}, M.},
  year = {2003},
  journal = {Moscow Mathematical Journal},
  volume = {3},
  number = {1},
  pages = {1--36},
  issn = {16093321, 16094514},
  doi = {10.17323/1609-4514-2003-3-1-1-36},
}

@article{neemanConnectionTheoryLocalization1992,
  title = {The Connection between the $K$-Theory Localization Theorem of {{Thomason}}, {{Trobaugh}} and {{Yao}} and the Smashing Subcategories of {{Bousfield}} and {{Ravenel}}},
  author = {Neeman, Amnon},
  year = {1992},
  journal = {Annales scientifiques de l'\'Ecole normale sup\'erieure},
  volume = {25},
  number = {5},
  pages = {547--566},
  issn = {0012-9593, 1873-2151},
  doi = {10.24033/asens.1659},
  langid = {english},
}

@article{balmerSpectrumPrimeIdeals2005,
  title = {The Spectrum of Prime Ideals in Tensor Triangulated Categories},
  author = {Balmer, Paul},
  year = {2005},
  month = {11},
  volume = {2005},
  number = {588},
  pages = {149--168},
  publisher = {{De Gruyter}},
  issn = {1435-5345},
  doi = {10.1515/crll.2005.2005.588.149},
  chapter = {Journal f\"ur die reine und angewandte Mathematik},
  langid = {english},
 }

@article{mathewThickSubcategoryTheorem2015a,
  title = {A Thick Subcategory Theorem for Modules over Certain Ring Spectra},
  author = {Mathew, Akhil},
  year = {2015},
  month = {7},
  journal = {Geometry \& Topology},
  volume = {19},
  number = {4},
  pages = {2359--2392},
  publisher = {{Mathematical Sciences Publishers}},
  issn = {1364-0380},
  doi = {10.2140/gt.2015.19.2359},
}

@article{neemanChromaticTower1992,
  title = {The Chromatic Tower for $D(R)$},
  author = {Neeman, Amnon and B{\"o}kstedt, Marcel},
  year = {1992},
  month = {7},
  journal = {Topology},
  volume = {31},
  number = {3},
  pages = {519--532},
  issn = {0040-9383},
  doi = {10.1016/0040-9383(92)90047-L},
}

@article{hoveyAxiomaticStableHomotopy1997a,
  title = {Axiomatic Stable Homotopy Theory},
  author = {Hovey, Mark and Palmieri, John H. and Strickland, Neil P.},
  year = {1997},
  journal = {Memoirs of the American Mathematical Society},
  volume = {128},
  number = {610},
  issn = {0065-9266, 1947-6221},
  doi = {10.1090/memo/0610},
  langid = {english},

}

@article{landKtheoryPullbacks2019,
  title = {On the $K$-Theory of Pullbacks},
  author = {Land, Markus and Tamme, Georg},
  year = {2019},
  month = {11},
  journal = {Annals of Mathematics},
  volume = {190},
  number = {3},
  pages = {877--930},
  publisher = {Department of Mathematics of Princeton University},
  issn = {0003-486X, 1939-8980},
  doi = {10.4007/annals.2019.190.3.4},
  urldate = {2025-05-21},
}

@article{nikolausTopologicalCyclicHomology2018a,
  title = {On Topological Cyclic Homology},
  author = {Nikolaus, Thomas and Scholze, Peter},
  year = {2018},
  month = {12},
  journal = {Acta Mathematica},
  volume = {221},
  number = {2},
  pages = {203--409},
  publisher = {Institut Mittag-Leffler},
  issn = {0001-5962, 1871-2509},
  doi = {10.4310/ACTA.2018.v221.n2.a1},
  urldate = {2022-11-14},
}

@misc{calmesHermitianKtheoryStable2025,
  title = {Hermitian {{K-theory}} for Stable $\infty$-Categories {{II}}: {{Cobordism}} Categories and Additivity},
  shorttitle = {Hermitian {{K-theory}} for Stable $\infty$-Categories {{II}}},
  author = {Calm{\`e}s, Baptiste and Dotto, Emanuele and Harpaz, Yonatan and Hebestreit, Fabian and Land, Markus and Moi, Kristian and Nardin, Denis and Nikolaus, Thomas and Steimle, Wolfgang},
  year = {2025},
  month = {4},
  number = {arXiv:2009.07224},
  eprint = {2009.07224},
  primaryclass = {math},
  publisher = {arXiv},
  doi = {10.48550/arXiv.2009.07224},
  urldate = {2025-05-23},
  archiveprefix = {arXiv},
}

@article{KockPitsch17,
  title = {Hochster Duality in Derived Categories and Point-Free Reconstruction of Schemes},
  volume = {369},
  issn = {0002-9947},
  number = {1},
  journal = {Trans. Amer. Math. Soc.},
  author = {Kock, Joachim and Pitsch, Wolfgang},
  year = {2017},
  pages = {223--261},
}

@article{thomasonClassificationTriangulatedSubcategories1997,
  title = {The Classification of Triangulated Subcategories},
  author = {Thomason, R. W.},
  year = {1997},
  month = {1},
  journal = {Compositio Mathematica},
  volume = {105},
  number = {1},
  pages = {1--27},
  issn = {1570-5846, 0010-437X},
  doi = {10.1023/A:1017932514274},
  urldate = {2025-08-15},
  langid = {english},
}

@misc{sanders2025tensortriangulargeometryfully,
      title={The tensor triangular geometry of fully faithful functors}, 
      author={Beren Sanders},
      year={2025},
      eprint={2508.02105},
      archivePrefix={arXiv},
      primaryClass={math.AT},
      url={https://arxiv.org/abs/2508.02105}, 
}

@article{Benson1997,
author = {Benson, D. and Carlson, Jon and Rickard, Jeremy},
journal = {Fundamenta Mathematicae},
language = {eng},
number = {1},
pages = {59-80},
title = {Thick subcategories of the stable module category},
url = {http://eudml.org/doc/212215},
volume = {153},
year = {1997},
}

@unpublished{TFEndo,
  author = {Chedalavada, Anish},
  title = {Torsion-Free Endotrivial Modules via Spectral Algebraic Geometry},
  note={In preparation.},
}

@article{balmerGeometryPermutationModules2025,
  title = {The Geometry of Permutation Modules},
  author = {Balmer, Paul and Gallauer, Martin},
  year = {2025},
  month = {9},
  journal = {Inventiones mathematicae},
  volume = {241},
  number = {3},
  pages = {841--928},
  publisher = {Springer Berlin Heidelberg},
  issn = {1432-1297},
  doi = {10.1007/s00222-025-01350-z},
  urldate = {2025-10-10},
  copyright = {2025 The Author(s), under exclusive licence to Springer-Verlag GmbH Germany, part of Springer Nature},
  langid = {english},
}

@misc{dubeyBalmerSpectrumIntegral2025,
  title = {The {{Balmer}} Spectrum of Integral Permutation Modules},
  author = {Dubey, Umesh V. and G{\'o}mez, Juan Omar},
  year = {2025},
  month = {11},
  number = {arXiv:2507.05892},
  eprint = {2507.05892},
  primaryclass = {math},
  publisher = {arXiv},
  doi = {10.48550/arXiv.2507.05892},
  urldate = {2025-10-10},
  archiveprefix = {arXiv},
 }

@inproceedings{Balmer2011,
  title = {Tensor Triangular Geometry},
  url = {http://dx.doi.org/10.1142/9789814324359_0041},
  DOI = {10.1142/9789814324359_0041},
  booktitle = {Proceedings of the International Congress of Mathematicians 2010 (ICM 2010)},
  publisher = {Published by Hindustan Book Agency (HBA),  India. WSPC Distribute for All Markets Except in India},
  author = {Balmer,  Paul},
  year = {2011},
  month = 6,
  pages = {85–112}
}

@incollection{hopkinsGlobalMethodsHomotopy1987,
  title = {Global Methods in Homotopy Theory},
  booktitle = {Homotopy {{Theory}}: {{Proceedings}} of the {{Durham Symposium}} 1985},
  author = {Hopkins, Michael J.},
  editor = {Rees, E. and Jones, J. D. S.},
  year = {1987},
  series = {London {{Mathematical Society Lecture Note Series}}},
  pages = {73--96},
  publisher = {Cambridge University Press},
  address = {Cambridge},
  doi = {10.1017/CBO9781107325746.005},
  urldate = {2025-06-25},
  isbn = {978-0-521-33946-9},
}

@misc{mathew2015torusactionsstablemodule,
      title={Torus actions on stable module categories, Picard groups, and localizing subcategories}, 
      author={Akhil Mathew},
      year={2015},
      eprint={1512.01716},
      archivePrefix={arXiv},
      primaryClass={math.RT},
      url={https://arxiv.org/abs/1512.01716}, 
}

@book{Johnstone1986-th,
  title     = {Stone Spaces},
  author    = {Johnstone, Peter T},
  publisher = {Cambridge University Press},
  series    = {Cambridge Studies in Advanced Mathematics},
  volume    = {3},
  ISBN      = {9780521337793},
  month     = {8},
  year      = {1986},
  address   = {Cambridge, England},
}

@article{Blumberg2026,
  title = {Chromatic convergence for the algebraic K‐theory of the sphere spectrum},
  volume = {113},
  ISSN = {1469-7750},
  url = {http://dx.doi.org/10.1112/jlms.70467},
  DOI = {10.1112/jlms.70467},
  number = {4},
  journal = {Journal of the London Mathematical Society},
  publisher = {Wiley},
  author = {Blumberg,  Andrew J. and Mandell,  Michael A. and Yuan,  Allen},
  year = {2026},
  month = Apr 
}

\end{document}